\documentclass[10pt,reqno]{amsart}
\RequirePackage[OT1]{fontenc}
\RequirePackage{amsthm,amsmath}
\RequirePackage[numbers]{natbib}
\RequirePackage[colorlinks,linkcolor=blue,citecolor=blue,urlcolor=blue]{hyperref}
\usepackage{amssymb}
\usepackage{anysize}
\usepackage{hyperref}
\usepackage{extarrows}
\usepackage{multirow}
\usepackage{natbib}
\usepackage{stmaryrd}
\usepackage{color}
\usepackage{amsmath}
\usepackage{enumitem}
\usepackage{mathrsfs}
\usepackage[dvipsnames]{xcolor}
\usepackage{bbm}

\usepackage{comment}
 
\numberwithin{equation}{section}

\theoremstyle{plain} 
\newtheorem{thm}{Theorem}[section]
\newtheorem{lem}[thm]{Lemma}

\newtheorem{prop}[thm]{Proposition}
\newtheorem{asm}[thm]{Assumption}
\newtheorem{de}[thm]{Definition}
\newtheorem*{thm*}{Theorem}

\theoremstyle{remark}
\newtheorem{rmk}[thm]{Remark}
\newtheorem{ex}[thm]{Example}

\newcommand{\ii}{\mathrm{i}}

\newcommand{\cf}{\emph{c.f., }}

\newcommand{\wt}{\widetilde}
\newcommand{\diag}{\mathrm{diag}}

\newcommand{\bs}{\boldsymbol}

\newcommand{\la}{\langle}
\newcommand{\ra}{\rangle}

\newcommand{\Cg}{\color{magenta}}

\newcommand{\rf}[1]{(\ref{#1})}
\newcommand{\mb}[1]{\mathbf{#1}}
\newcommand{\sn}{\sqrt{N}}
\newcommand{\mG}{\mathcal{G}}

\newcommand{\mP}{\mathcal{P}}

\newcommand {\mbE}{\mathbb{E}}

\newcommand{\onh}{O_\prec(N^{-\frac12})}
\newcommand {\bu}{{\mb{u}}}
\newcommand {\bv}{{\mb{v}}}
\newcommand {\bw}{{\mb{w}}}

\newcommand {\lb}{\llbracket}
\newcommand {\rb}{\rrbracket}

\renewcommand{\mathbf}[1]{\bs{#1}}
 
\begin{document}

\begin{center}
\large\bf
{ Principal components of  spiked covariance matrices in the supercritical regime } 
\end{center}

\renewcommand{\thefootnote}{\fnsymbol{footnote}}	
\vspace{1cm}
\begin{center}
 \begin{minipage}[t]{0.3\textwidth}
\begin{center}
Zhigang Bao\footnotemark[1]  \\
\footnotesize {Hong Kong University of Science and Technology}\\
{\it mazgbao@ust.hk}
\end{center}
\end{minipage}
\hspace{8ex}
\begin{minipage}[t]{0.3\textwidth}
\begin{center}
Xiucai Ding\footnotemark[2]  \\ 
\footnotesize {Duke University}\\
{\it xiucai.ding@duke.edu}
\end{center}
\end{minipage}

\vspace{3ex}

\hspace*{3ex}\begin{minipage}[t]{0.3\textwidth}
 \begin{center}
Jingming Wang\footnotemark[4]\\
\footnotesize 
{Hong Kong University of Science and Technology}\\
{\it jwangdm@connect.ust.hk}
\end{center}
\end{minipage}
\hspace{8ex}
\begin{minipage}[t]{0.3\textwidth}
 \begin{center}
Ke Wang\footnotemark[3]\\
\footnotesize 
{Hong Kong University of Science and Technology}\\
{\it kewang@ust.hk}
\end{center}
\end{minipage}
\footnotetext[1]{Supported by Hong Kong RGC grant ECS 26301517, GRF 16300618, and GRF 16301519.}
\footnotetext[4]{Supported by Hong Kong RGC grant ECS 26301517, GRF 16300618, and GRF 16301519.}
\footnotetext[3]{Supported by Hong Kong RGC grant GRF 16301618 and HKUST Initiation Grant IGN16SC05.}

\renewcommand{\thefootnote}{\fnsymbol{footnote}}	

\end{center}
\vspace{1cm}
\begin{center}
 \begin{minipage}{0.8\textwidth} \footnotesize{ In this paper, we study the asymptotic behavior of the extreme eigenvalues and  eigenvectors of the spiked covariance matrices, in the supercritical regime. Specifically, we  derive the joint distribution of the extreme eigenvalues and the generalized components of their associated eigenvectors in this regime. }
\end{minipage}
\end{center}

\vspace{2mm}
 
 {\small
\footnotesize{\noindent\textit{Date}: \today}
 
 \vspace{2mm}

 }

\thispagestyle{plain}

\section{Introduction}
In this paper, we consider the sample covariance matrices of the form
\begin{align}
Q=TXX^*T^*,
\end{align}
where $T$ is a $M\times M$ deterministic matrix and $X$ is a $M\times N$ random matrix with independent entries. Further, we assume that the population covariance matrix $\Sigma:=TT^*$ admits the following form
\begin{align}
\Sigma=I_M+S, \label{19071901}   
\end{align}
where $S$ is a fixed-rank deterministic Hermitian matrix.

Throughout the paper, we make the following assumptions.
\begin{asm}\label{assumption} 
\hfill \break

\noindent (i)({\it On dimensions}):  We assume that $M\equiv M(N)$ and $N$ are comparable and there exist constants $\tau_2>\tau_1>0$ such that 
\begin{align*}
y\equiv y_N=M/N\in (\tau_1,\tau_2). 
\end{align*}

\noindent (ii)({\it On $S$}): We assume that $S$ admits the following spectral decomposition
\begin{align}
S=\sum_{i=1}^r d_i \bv_i\bv_i^*, \label{19071902}
\end{align}
where $r\geq 1$ is a fixed integer. Here $d_1>\cdots>d_r >0$ are the ordered eigenvalues of $S$, and $\bv_i=(v_{i1},\ldots, v_{iM})^T$'s are the associated unit eigenvectors.

\noindent (iii)({\it On $X$}): For the matrix $X=(x_{ij})$, we assume that the entries $x_{ij}\equiv x_{ij}(N)$ are real random variables satisfying
\begin{align*}
\mathbb{E}x_{ij}=0,\qquad \mathbb{E}x_{ij}^2=1/N.
\end{align*}
Moreover, we assume the existence of  large moments, {\it i.e,} for any integer $p\geq3$, there exists a constant $C_p>0$, such that
\begin{align*}
\mathbb{E}\vert\sn x_{ij}\vert^p\leq C_p<\infty.
\end{align*}
We further assume that all $\sqrt{N}x_{ij}$'s possess the same $3$rd and $4$th cumulants, which are denoted by $\kappa_3$ and $\kappa_4$ respectively. 
\end{asm}

We denote by $\mu_1\geq \cdots\geq \mu_M$ the ordered eigenvalues of $Q$ and  $\mb{\xi}_i$ the unit eigenvector associated with $\mu_i$.   The matrix model $Q$ is usually referred to  as the spiked covariance matrix  in literature.  In the context of Random Matrix Theory (RMT), this model was first studied by Johnstone in \cite{Johnstone}. The primary interest of the spiked model $Q$ lies in the asymptotic behavior of a few largest  $\mu_i$'s and the associated $\mb{\xi}_i$'s when $N$ is large, under various assumptions of $d_i$'s and $\mathbf{v}_i$'s. Significant progress has been made on this topic, in the last few years.  It has been well known since the seminal work of Baik, Ben Arous and P\'{e}ch\'{e} \cite{BBP}  that the largest eigenvalues $\mu_i$'s undergo a phase transition (BBP transition) w.r.t. the size of $d_i$'s. On the level  of the first order limit,  when $d_i>\sqrt{y}$, the eigenvalue $\mu_i$  jumps out of the support of the Marchenko-Pastur law (MP law) and converges to a limit determined by $d_i$, while  in the case of  $d_i\leq \sqrt{y}$ it sticks to the right end of the MP law  $(1+\sqrt{y})^2$. On the level of the second order fluctuation, it was revealed in \cite{BBP} that a phase transition  for  $\mu_i$ takes place in the regime $d_i-\sqrt{y}\sim N^{-\frac13}$. Specifically, if $d_i-\sqrt{y}\ll N^{-\frac13}$ (subcritical regime), the eigenvalue $\mu_i$ still admits the Tracy-Widom type distribution; if $d_i-\sqrt{y}\gg N^{-\frac13}$ (supercritical regime), the eigenvalue  $\mu_i$ is asymptotically Gaussian; while if $d_i-\sqrt{y}\sim N^{-\frac13}$ (critical regime), the limiting distribution of the eigenvalue $\mu_i$ is a mixture of Tracy-Widom and Gaussian.  The works \cite{Johnstone} and \cite{BBP} are on real and complex spiked Gaussian covariance matrices respectively. On extreme eigenvalues, further study for more generally distributed  covariance matrices can be found in \cite{BaikS, CDM16, BGM11, Paul, BY08, BY12, bloemendal2016principal, DY19}. The limiting behavior of the extreme eigenvalues  have also been studied for various related models, such as  the finite-rank deformation of Wigner matrices \cite{CDM16, BGM11, CDF09,  CDF12,   FP07, KY13, KY14,  Peche06, PRS13},  the signal-plus-noise model \cite{BGN12, LV, Ding}, the general spiked $\beta$ ensemble \cite{BVirag13, BVirag16}, and also the finite-rank deformation of general unitary/orthogonal invariant matrices \cite{BGN11, BBCF, BBCFAOP}.  

In contrast, the study on the limiting behavior of the eigenvectors associated with the extreme eigenvalues is much less.   On the level of the first order limit, it is known that the $\mathbf{\xi}_i$'s  are delocalized and purely noisy in the subcritical regime, but has a bias on the direction of $\mathbf{v}_i$ in the supercritical regime.  We refer to \cite{BGN11, BGN12, Capitaine17, Ding, Paul, bloemendal2016principal, DY19} for more details of such a phenomenon. It was recently noticed in \cite{ bloemendal2016principal} that a $d_i$ close to the critical point can cause a bias even for the non-outlier eigenvectors. On the level of the second order fluctuation, it was proved in \cite{ bloemendal2016principal} that the eigenvectors are asymptotically Gaussian in the subcritical regime, for the spiked covariance matrices. In the supercritical regime, a non-universality phenomenon was shown in \cite{CDM18} and \cite{BDW} for the eigenvector distribution for the finite-rank deformation of Wigner matrices and the signal-plus-noise model, respectively. The non-universality phenomenon in the supercritical regime has been previously observed in \cite{CDF09, KY13, KY14} for the extreme eigenvalues of the  finite-rank deformation of Wigner matrices.  Here we also refer to \cite{MJMY, JY, FFHL} for related study on the extreme eigenstructures of various finite-rank deformed models from more statistical perspective.   

In this paper, we will establish the joint distribution of the extreme eigenvalues and the associated  eigenvectors for the spiked covariance matrices, in the supercritical regime.  This is the primary goal of the Principal Component Analysis from statistics point of view. More specifically, in this paper, we are interested in the joint distribution of the largest $\mu_i$'s and the  generalized component of the top eigenvectors, i.e., the projections of those eigenvectors onto a general direction. More specifically, let $\mb{w}\in S^{M-1}_\mathbb{R}$ be any deterministic unit vector. We will study the limiting distribution of $(\mu_i, |\langle\mb{w},\mb{\xi}_i\rangle|^2)$ in the supercritical regime.

Since we will focus on the supercritical regime, we  further make the following assumption. For brevity, in the sequel, we use the notation $\lb 1, m\rb:=\{1,\cdots,m\}$.
\begin{asm} \label{supercritical}
There exists a fixed $\delta>0$ and an integer $r_0\in \lb1, r\rb$ such that for $i\in \lb1, r_0\rb$,
\begin{align}
d_i\geq y^{1/2}+\delta\quad\text{and} \quad \min_{j\neq i\in  \lb 1, r_0\rb}|d_j-d_i|\geq \delta. \label{asd}
\end{align}
\end{asm}
\begin{rmk}
The first inequality ensures the existence of the supercritical regime,  and the second inequality ensures that those outlying eigenvalues are well-separated from each other. In both inequalities, the fixed constant $\delta$ may be  replaced by smaller $N$-dependent ones. One can also consider the case that $d_i$'s are not simple, i.e., the multiplicity of certain $d_i$ is larger than 1.  We refer to \cite{bloemendal2016principal} for related discussion. Further extension along this direction will be considered in the future work. 
\end{rmk}

Our main results will be stated in Theorem \ref{rankrmainthm}, after necessary notations are introduced. 
For simplicity,   we will work with the setting 
\begin{align}
T= \Sigma^{\frac12}. \label{19071802}
\end{align} 
But our results hold for more general $T$ satisfying $\Sigma=TT^*$. Extension along this direction has been discussed in Section 8 of \cite{bloemendal2016principal}. We refer to Remark \ref{rmk.extension} for more details. 

For any $\mathbf{w}\in S_{\mathbb{R}}^{M-1}$, we 
do the decomposition 
\begin{align}
\mb{w}=\sum_{j=1}^r \langle\mb{w}, \bv_j\rangle \bv_j+\bu, \quad \text{where $ \bu\in \mathrm{Span}\{\bv_1\cdots,\bv_r\}^\perp$}. \label{19071820}
\end{align}
Hereafter,  we take (\ref{19071820}) as the definition of $\mathbf{u}$. 
 Further, we introduce the following two shorthand notations
 \begin{align}
& \mb{\varsigma}_i:=\frac{2\sqrt{1+d_i}(d_i^2-y)}{d_i^2(d_i+y)}\Big(\sum_{j\neq i}\la\mb{w},\bv_j\ra\frac{d_i\sqrt{d_j+1}}{d_i-d_j}\bv_j+\bu\Big), \nonumber\\
& \widehat{\bv}_i:=\la \mb{w},\mb{v}_i\ra\frac{y(1+d_i)}{d_i^2(d_i+y)}\Big(1+\frac{d_i(d_i+1)}{d_i+y}\Big)\bv_i. \label{def_eta_hatvi}
 \end{align}
 For any vectors $\mathbf{a}_\gamma=(a_\gamma(i))\in \mathbb{R}^M, \gamma=1,2,3$, we set the notations 
\begin{align}
&\mathbf{s}_{k}(\mathbf{a}_1):= \sum_{j=1}^M a_1(j)^k, \quad\mathbf{s}_{k,l}(\mathbf{a}_1,\mathbf{a}_2):= \sum_{j=1}^M a_1(j)^k a_2(j)^l, \nonumber\\
&\mathbf{s}_{k,l,t}(\mathbf{a}_1,\mathbf{a}_2,\mathbf{a}_3) := \sum_{j=1}^M a_1(j)^k a_2(j)^l a_3(j)^t. \label{def:s12}
\end{align}

We adopt the notion of {\it stochastic domination}  introduced in \cite{EKY2013}, which provides a precise statement of the form ``$X_N$ is bounded by $Y_N$ up to a small power of $N$ with high probability''.
\begin{de}(Stochastic domination) Let
\begin{align*}
X=\big(X_N(u): N\in \mathbb{N}, u\in U_N\big), \quad Y=\big(Y_N(u): N\in \mathbb{N}, u\in U_N\big)
\end{align*}
 be two families of nonnegative random variables, where $U_N$ is a possibly $N$-dependent parameter set. We say that $X$ is bounded by $Y$, uniformly in $u$, if for all small $\epsilon>0$ and large $\phi>0$, we have
 \begin{align*}
 \sup_{u\in U_N}\mathbb{P}\big(X_N(u)>N^{\epsilon}Y_N(u)\big)\leq N^{-\phi}
 \end{align*}
 for large $N\geq N_0(\epsilon, \phi)$. Throughout the paper, we use the notation $X=O_\prec(Y)$ or $X\prec Y$ when $X$ is stochastically bounded by $Y$ uniformly in $u$.  
 
 In addition,  we also say that an $n$-dependent event $E\equiv E(n)$ holds with high probability if, for any large $\varphi>0$,  
\begin{align*}
\mathbb{P}(E) \geq 1- n^{-\varphi},
\end{align*}
for sufficiently large $n \geq n_0( \varphi)$.
\end{de}

It is  convenient to introduce the following notion of convergence in distribution.
\begin{de} \label{defn_asymptotic}
Two sequences of random vectors, $\mathsf X_n \in \mathbb{R}^k$ and $\mathsf Y_n \in \mathbb{R}^k$, $n\geq 1$,  are \emph{asymptotically equal in distribution}, denoted as $\mathsf X_n \simeq \mathsf Y_n,$ if they are tight and satisfy
\begin{equation*}
\lim_{n \rightarrow \infty} \big( \mathbb{E}f(\mathsf X_n)-\mathbb{E}f(\mathsf Y_n) \big)=0
\end{equation*}   
for any bounded continuous function $f:\mathbb{R}^k\to \mathbb{R}$. 
\end{de}

Furthermore, for a vector $\mathbf{a}$, we use $\|\mathbf{a}\|$  to denote its   $\ell^2$-norm. Let $\{\mathbf{e}_i\}_{i=1}^M$  be the standard basis of $\mathbb{R}^M$. For a random variable $\xi$, we denote its $l$th cumulant by $\kappa_l(\xi)$. 

With the above notations, we can now state our main Theorem.

 \begin{thm}\label{rankrmainthm}
Suppose that Assumptions \ref{assumption},  \ref{supercritical}
 and the setting (\ref{19071802}) hold.  Then for any deterministic unit vector $\mb{w}\in S^{M-1}_{\mathbb{R}}$  and  $i\in \lb 1, r_0\rb$, there exist random variables $\Upsilon_i, \Theta_i^{\mb{w}}$ and $\Lambda_i^{\mb{w}}$ such that 
\begin{align} \label{thmdecompeigenv}
\mu_i=1+d_i+y+\frac{y}{d_i}+\frac{1}{\sqrt{N}}\Upsilon_i+O_\prec\Big(\frac 1N\Big), 
\end{align}
and 
 \begin{align}
 |\langle \mb{w}, \mb{\xi}_i\rangle|^2&=\frac{d_i^2-y}{d_i(d_i+y)}|\la \mb{w},\mb{v}_i\ra|^2+\frac{\la \mb{w},\mb{v}_i\ra}{\sn}\Theta_i^{\mb{w}}+\frac{1}{N}(\Lambda_i^{\mb{w}})^2\nonumber\\
 &\qquad\qquad +O_\prec\Big(\frac{|\la \mb{w},\mb{v}_i\ra|}{N}\Big)+O_\prec(N^{-\frac32}), \label{rankrdecompv3}
 \end{align}
 and
\begin{align}
\big(\Upsilon_i,\Theta_i^{\mb{w}}, \Lambda_i^{\mb{w}}\big)\simeq \mathcal{N}\big(0,\mathcal{V}_i^{\mb{w}}\big).
 \end{align}
 Here, $\mathcal{N}\big(0,\mathcal{V}_i^{\mb{w}}\big)$ represents a Gaussian vector with mean $0$ and  covariance matrix $\mathcal{V}_i^{\mb{w}}$ defined entrywise by
 \begin{align}
\mathcal{V}_i^{\mb{w}}(1,1)=&\frac{(1+d_i)^2(d_i^2-y)^2}{d_i^4}\Big(\frac{2d_i^2}{d_i^2-y}+\kappa_4\mathbf{s}_4(\bv_i)\Big),\nonumber\\
  \mathcal{V}_i^{\mb{w}}(1,2)=&\langle \bw, \bv_i\rangle\frac{2y(1+d_i)^3}{d_i(d_i+y)^2}+\frac{(1+d_i)(d_i^2-y)}{d_i^2}\kappa_4\mathbf{s}_{1,3}(\widehat{\bv}_i+\mb{\varsigma}_i,\bv_i), \nonumber\\
\mathcal{V}_i^{\mb{w}}(1,3)=&-\frac{d_i+1}{2d_i}\sqrt{\frac{(d_i+y)(d_i^2-y)}{{d_i}}}\mathbf{s}_{1,3}(\mb{\varsigma}_i,\bv_i),\nonumber\\
\mathcal{V}_i^{\mb{w}}(2,2)=&\frac{d_i^2}{d_i^2-y}\Vert\widehat{\bv}_i+\mb{\varsigma}_i\Vert^2+\kappa_4\mathbf{s}_{2,2}(\widehat{\bv}_i+\mb{\varsigma}_i,\bv_i)\nonumber\\
 &+\langle \bw, \bv_i\rangle\frac{d_iy+d_i+2y}{(d_i+y)^2}\mathbf{s}_{1,1}(\widehat{\bv}_i,\bv_i)+\langle \bw, \bv_i\rangle^2\frac{y(1+d_i)(d_i^2-y)}{d_i^2(d_i+y)^3},\nonumber\\
\mathcal{V}_i^{\mb{w}}(2,3)= &\frac{1}{2}\sqrt{\frac{d_i(d_i+y)}{d_i^2-y}} \Big(\frac{d_i^2}{d_i^2-y}\Vert\mb{\varsigma}_i\Vert^2+\kappa_4\mathbf{s}_{2,2}(\mb{\varsigma}_i,\bv_i)-\kappa_4\mathbf{s}_{1,1,2}(\mb{\varsigma}_i,\widehat{\bv}_i,\bv_i)\Big),\nonumber\\
\mathcal{V}_i^{\mb{w}}(3,3) = &\frac{d_i(d_i+y)}{4(d_i^2-y)}\Big(\frac{d_i^2}{d_i^2-y}\Vert\mb{\varsigma}_i\Vert^2+\kappa_4\mathbf{s}_{2,2} (\mb{\varsigma}_i, \bv_i)  \Big). \label{V_22}
 \end{align}
 \end{thm}
 
 \begin{rmk} Here we remark that in the supercritical regime, a generalized CLT for the eigenvalues  has been established in \cite{BY08} previously.   
On eigenvectors, in the supercritical regime, a large deviation estimate for the generalized  components of the eigenvectors has also been given in \cite{bloemendal2016principal}. In addition, the Gaussianity of  the generalized components of eigenvectors in the subcritical regime has been proved in  \cite{bloemendal2016principal}. 
 \end{rmk}
 
 \begin{rmk} In Assumption \ref{assumption}, we assume that the third and fourth cumulants of all $\sqrt{N}x_{ij}$'s are the same. But our proof can be directly adapted to the more general setting that the third and fourth cumulants of  $\sqrt{N}x_{ij}$'s are $(ij)$-dependent. 
 \end{rmk}
 
 \begin{rmk}\label{rmk.extension} In (\ref{19071802}), we assumed $T=\Sigma^{\frac12}$. But our result holds under much more general assumption on $T$. We can indeed extend our result to the matrix $Q=TXX^*T^*$ with $(M+k)\times N$ random matrix $X$ and $M\times (M+k)$ matrix $T$. As long as $k\in \mathbb{N}$ is  fixed and $TT^*=\Sigma$ satisfying (\ref{19071901}), our result remains true. Such an extension  has been discussed in Section 8 of \cite{bloemendal2016principal}. The discussion in \cite{bloemendal2016principal} relies on the rewriting  $Q=TXX^*T^*=\Sigma^{\frac12}YY^*\Sigma^{\frac12}$, where $Y= (I_M\;0)OX$. Here $0$ is the $M\times k$ zero matrix and $O$ is some $(M+k)\times (M+k)$ orthogonal matrix.  It will be clear that all our arguments in the proof of Theorem \ref{rankrmainthm} are valid, as long as the isotropic local law of the matrix $XX^*$ (\cf Theorem \ref{isotropic}) holds. So here it would be sufficient to have the isotropic local law for the matrix $YY^*$, which has been demonstrated in Theorem 8.1 of  \cite{bloemendal2016principal}. 
 \end{rmk}

In the sequel, we provide some simple examples with  special choices of $\mb{w}$. 
 \begin{ex}
 Let $\mb{w}=\bv_i$. The expansion \eqref{rankrdecompv3} can be simplified to 
 \begin{align}
 |\la\bv_i, \mb{\xi}_i\ra|^2=\frac{d_i^2-y}{d_i(d_i+y)}+\frac{1}{\sn}\Theta_i+O_\prec(\frac{1}{N}),
 \end{align}
and $\Theta_i\simeq\mathcal{N}(0,\mathcal{V}_{i}(2,2))$ with explicit expression
 \begin{align}
\mathcal{V}_{i}(2,2)=\frac{d_i^2}{d_i^2-y}\Vert\widehat{\bv}_i\Vert^2+\frac{d_iy+d_i+2y}{(d_i+y)^2}\mathbf{s}_{1,1}(\widehat{\bv}_i,\bv_i)+\frac{y(1+d_i)(d_i^2-y)}{d_i^2(d_i+y)^3}+\kappa_4\mathbf{s}_{2,2}(\widehat{\bv}_i,\bv_i).
 \end{align}
Here 
 \begin{align*}
 \widehat{\bv}_i=\frac{y(1+d_i)}{d_i^2(d_i+y)}\Big(1+\frac{d_i(d_i+1)}{d_i+y}\Big)\bv_i.
 \end{align*}
 \end{ex}

\begin{ex}
Let $\mb{w}\in \{\mb{v}_i\}^{\perp}$. In this case, we have $\widehat{\bv}_i=0$. We then get from \eqref{rankrdecompv3} that
\begin{align}
 |\la \mb{w}, \mb{\xi}_i\ra|^2=\frac{1}{N}(\Lambda_i^{\mb{w}})^2+O_\prec(N^{-\frac32}),
\end{align}
and $\Lambda_i^{\mathbf{w}}\simeq\mathcal{N}(0,\mathcal{V}_i^{\mb{w}}(3,3))$ where   
\begin{align*}
\mathcal{V}_i^{\mb{w}}(3,3)=\frac{d_i(d_i+y)}{4(d_i^2-y)}\Big(\frac{d_i^2}{d_i^2-y}\Vert\mb{\varsigma}_i\Vert^2+\kappa_4\mathbf{s}_{2,2} (\mb{\varsigma}_i, \bv_i)  \Big),
\end{align*}
with $\mb{\varsigma}_i$ defined in the first equation of \rf{def_eta_hatvi}.
\end{ex}

\vspace{2ex}
{\bf Organization:} The paper is organized as the following: In Section \ref{s. pre}, we introduce some basic notions and preliminary results for later discussion. Section \ref{secgfr} is devoted to the Green function representations of our eigenvalue and eigenvector statistics. Then in Section \ref{s.proof of thm}, we prove our main result, Theorem \ref{rankrmainthm}, based on a key technical recursive moment estimate for  some Green function statistics, see Proposition \ref{recursivemP}.  The proof of Proposition \ref{recursivemP} is then postponed to Section \ref{sec:prop}.  In addition, in Appendix \ref{s.derivative of G}, we collect some basic formulas concerning the derivatives of Green function, for the convenience of the reader.

\section{Preliminaries} \label{s. pre}
In this section, we collect some basic notions and preliminary results which will be used in the proof of our main theorem. A key technical input is the isotropic local law from \cite{bloemendal2014isotropic, knowles2017anisotropic}.

\subsection{Basic notions}\label{basicnotations}

In the sequel,  we denote the Green function of $Q$ by
\begin{align}
G(z):=(Q-z)^{-1}, \quad z\in \mathbb{C}^+. \nonumber
\end{align}
The matrix $Q$ can be regarded as a finite-rank perturbation of the matrix $H:=XX^*$. In the sequel, we also need to consider $\mathcal{H}:=X^*X$ which shares the same non-zero eigenvalues with $H$. We further denote the Green functions of $H$ and $\mathcal{H}$ respectively  by
\begin{align}\label{def:green}
\mG_1(z):=(XX^*-z)^{-1}&, \qquad \mG_2(z):=(X^*X-z)^{-1}, \quad z\in \mathbb{C}^+.  
\end{align}
We also denote the  normalized traces of $\mG_1(z)$ and $\mG_2(z)$ by 
\begin{align*}
m_{1N}(z):=\frac{1}{M}\text{Tr}\mG_1(z)=\int(x-z)^{-1}\,dF_{1N}(x)&, \quad m_{2N}(z):=\frac{1}{N}\text{Tr}\mG_2(z)= \int(x-z)^{-1}\,dF_{2N}(x),
\end{align*}
where  $F_{1N}(x)$, $F_{2N}(x)$ are the empirical distributions of $H$ and $\mathcal{H}$ respectively, i.e.,
\begin{align}
F_{1N}(x):=\frac{1}{M}\sum_{i=1}^M \mathbbm{1}(\lambda_i(H)\leq x), \quad F_{2N}(x):=\frac{1}{N}\sum_{i=1}^N \mathbbm{1}(\lambda_i(\mathcal{H})\leq x). \nonumber
\end{align}
Here we used $\lambda_i(H)$ and $\lambda_i(\mathcal{H})$  to denote the $i$-th largest eigenvalue of $H$ and $\mathcal{H}$, respectively. 

It is well-known since \cite{MP67} that $F_{1N}(x)$ and $F_{2N}(x)$ converge weakly (a.s.) to the {\it Marchenko-Pastur} laws 
$\nu_{\text{MP},1}$ and $\nu_{\text{MP},2}$ (respectively) given below
\begin{align}
&\nu_{\text{MP},1}({\rm d}x):=\frac{1}{2\pi xy}\sqrt{\big((\lambda_+-x)(x-\lambda_-)\big)_+}{\rm d}x+(1-\frac{1}{y})_+\delta({\rm d}x),\nonumber\\
&\nu_{\text{MP},2}({\rm d}x):=\frac{1}{2\pi x}\sqrt{\big((\lambda_+-x)(x-\lambda_-)\big)_+}{\rm d}x+(1-y)_+\delta({\rm d}x), \label{19071801}
\end{align}
where $\lambda_{\pm}:=(1\pm \sqrt{y})^2$. 
Note that $m_{1N}$ and $m_{2N}$ can be regarded as the Stieltjes transforms of $F_{1N}$ and $F_{2N}$, respectively.  We further define their deterministic counterparts, i.e.,  Stieltjes transforms of $\nu_{\text{MP},1},\nu_{\text{MP},2}$,  by $m_1(z),m_2(z)$,  respectively, i.e.,
 \begin{align}
 m_1(z):=\int (x-z)^{-1}\nu_{\text{MP},1}({\rm d}x), \quad m_2(z):=\int (x-z)^{-1}\nu_{\text{MP},2}({\rm d}x). \nonumber
 \end{align}
 From the definition (\ref{19071801}), it is elementary to compute 
 \begin{align}
m_1(z)=\frac{1-y-z+\ii\sqrt{(\lambda_+-z)(z-\lambda_-)}}{2zy},  \quad
m_2(z)=\frac{y-1-z+\ii\sqrt{(\lambda_+-z)(z-\lambda_-)}}{2z}, \label{m1m2}
 \end{align}
where the square root is taken with a branch cut on the negative real axis. Equivalently, we can also characterize $m_1(z),m_2(z)$ as the unique solutions from $\mathbb{C}^+$ to $\mathbb{C}^+$ to the equations
\begin{align}
zym_1^2+[z-(1-y)]m_1+1=0, \qquad  zm_2^2+[z+(1-y)]m_2+1=0. \label{selfconeqt}
\end{align} 
Using \eqref{m1m2} and \eqref{selfconeqt}, one can easily derive the following identities  
\begin{align}
m_1=-\frac{1}{z(1+m_2)}, \quad 1+zm_1=\frac{1+zm_2}{y},\quad m_1\big((zm_2)'+1\big)=\frac{m_1'}{m_1}, \label{identitym1m2}
\end{align}
which will be used in the later discussions.

\subsection{Isotropic local law}
In this section, we state the isotropic local law from \cite{bloemendal2014isotropic, knowles2017anisotropic} together with some  consequences which will serve as the main technical inputs in the proofs. We first introduce the following domain. For a small (but fixed) $\tau>0$, we set
\begin{align}
{\mathscr{D}}\equiv {\mathscr{D}}(\tau):=\big\{z=E+\ii\eta\in\mathbb{C}:\lambda_{+}+\tau\leq E\leq\tau^{-1}, 0<\eta\leq\tau^{-1}\big\}. \label{19071810}
\end{align}
Conventionally,  for $a=1,2$,  we denote by $\mathcal{G}_a^l$ and $\mathcal{G}_a^{(l)}$  the $l$-th power of $\mathcal{G}_a$ and the $l$-th derivative of $\mathcal{G}_a$ w.r.t. $z$, respectively. 
With the above notation, we have 
\begin{thm} \label{isotropic}
Let  $\tau>0$ in (\ref{19071810}) be a small but fixed constant. Let  $\bu,\bv\in \mathbb{C}^M$ be two deterministic unit vectors. Suppose $X$ satisfies Assumption \ref{assumption}. Then,  for any given $l\in\mathbb{N}$ and $\alpha=1,2$, we have 
\begin{align}
&|\langle \bu,\mG_\alpha^{(l)}(z)\bv\rangle-m_\alpha^{(l)}(z)\langle \bu,\bv\rangle|=O_\prec\bigg(\sqrt{\frac{{\Im}m_\alpha(z)}{N\eta}}\bigg), \label{est_DG}\\
&|\langle \bu,X^*\mG_\alpha^{l}(z)\bv\rangle|=O_\prec\bigg(\sqrt{\frac{{\Im}m_\alpha(z)}{N\eta}}\bigg), \label{est_XG}\\
&|m_{\alpha N}(z)-m_\alpha(z)|=O_\prec\big(\frac{1}{N}\big) \label{est_m12N}, 
\end{align}
uniformly in $z\in{\mathscr{D}} $.
\end{thm}
\begin{rmk}
The case of $l=0$ is directly from the isotropic law in  Theorem 3.12 of \cite{bloemendal2014isotropic} and the anisotropic law in Theorem 3.7 of  \cite{knowles2017anisotropic}. For other $l\geq 1$, we can derive the estimate easily from the case $l=0$ by using Cauchy integral. We also remark here that the original isotropic local laws in \cite{bloemendal2014isotropic, knowles2017anisotropic} were stated in much larger domains which also include the bulk and  edge regimes of the MP law. But here we only need the result for the domain far away from the support of the MP law. 
\end{rmk}

Further,  in the following lemma, we collect some basic estimates of $m_1$ and $m_2$ which can be verified by elementary computations.
\begin{lem} \label{lem.19072501}
Recall the definition of $m_1$ and $m_2$ in \rf{m1m2}. For $a=1,2$, we have
\begin{align}
|m_{a}'(z)|\sim |m_{a}(z)|\sim 1 \quad {\Im} m_{a} \sim \frac{\eta}{\sqrt{\kappa+\eta}}, \label{estm1m2}
\end{align}
uniformly in $z\in{\mathscr{D}}$, where $\kappa=|E-\lambda_+|$.
\end{lem}
\begin{rmk} \label{boundrmk}
Following from Theorem \ref{isotropic} and Lemma \ref{lem.19072501}, we can easily have the boundedness of $\bu^*\mG_1^\alpha\bv $, $\bu^*X\mG_1^\alpha\bv$ and $\bu^*X\mG_1^\alpha X^*\bv$ for any positive integer $\alpha$  and deterministic unit vectors $\mathbf{u}$ and $\mathbf{v}$ of appropriate dimensions. Actually, in our regime, $z\in \mathscr{D}$, the boundedness of all these quantities follows more directly from the rigidity of the largest eigenvalue of $H$ (\cf (\ref{190726100})), which guarantees that $\|\mathcal{G}_1\|_{\text{op}}\leq C$ with high probability. 
\end{rmk}
Using the isotropic local law, one can also get the following result, which gives the location of the outlier and the  extremal non-outlier.
\begin{lem}\label{locationeig}
{\rm(}{\it Theorem 2.3} of \cite{bloemendal2016principal}{\rm)} Under Assumption \ref{assumption} and \rf{asd}, we have for $i\in \lb 1, r_0\rb$
\begin{align*}
|\mu_i-\theta(d_i)|\prec N^{-\frac12},\qquad |\mu_{r_0+1}-\lambda_+|\prec N^{-2/3},
\end{align*}
where
\begin{align}
\theta(z):=1+z+y+yz^{-1}, \qquad \text{for   } z\in \mathbb{C},\quad \Re z>\sqrt{y}. \label{def. of theta}
\end{align}
\end{lem}

Further, for the largest eigenvalue of $H$, denoted by $\lambda_1(H)$, we have the rigidity estimate
\begin{align}
|\lambda_1(H)-\lambda_+|\prec N^{-\frac23}.  \label{190726100}
\end{align}
We refer to Theorem 3.1 of \cite{PY14},  for instance.

\subsection{Auxiliary lemmas} 
The following \emph{cumulant expansion formula} plays a central role in our computation, whose proof can be found in \cite[Proposition 3.1]{LP09} or \cite[Section II]{KKP96}, for instance.

\begin{lem}\label{cumulantexpansion}
(Cumulant expansion formula) For a fixed $\ell\in \mathbb{N}$,  let $f\in C^{\ell+1}(\mathbb{R})$. Supposed $\xi$ is a centered random variable with finite moments  to order $\ell+2$.  Recall the notation $\kappa_k(\xi)$ for  the k-th cumulant of $\xi$.  Then we have  
\begin{align}
\mbE(\xi f(\xi))=\sum_{k=1}^{\ell}\frac{\kappa_{k+1}(\xi)}{k!}\mbE(f^{(k)}(\xi))+\mbE(r_{\ell}(\xi f(\xi))),
\end{align}
where the error term $r_{\ell}(\xi f(\xi))$ satisfies
\begin{align}
|\mbE(r_{\ell}(\xi f(\xi)))|\leq C_\ell\mbE(|\xi|^{\ell+2}){\rm{sup}}_{|t|\leq s}|f^{\ell+1}(t)|+C_\ell\mbE(|\xi|^{\ell+2}\mathbbm{1}(|\xi|>s)){\rm{sup}}_{t\in \mathbb{R}}|f^{\ell+1}(t)|  \label{remaining}
\end{align}
for any  $s>0$ and $C_\ell$ satisfied $C_\ell\leq (C\ell)^{\ell}/\ell!$ for some constant $C>0$.
\end{lem}

Next we collect some basic identities for the Green functions in (\ref{def:green}) without proof.

\begin{lem} \label{lemrelation}
For any integer $l\geq1$, we have
\begin{align}
&\mG_1^{l}=\frac{1}{(l-1)!}\frac{\partial^{l-1}\mG_1}{\partial z^{l-1}} = \frac{1}{(l-1)!} \mG_1^{(l-1)},\label{eq:basicG}\\   
&\mG_1^lXX^*=\mG_1^{l-1}+z\mG_1^{l}, \quad X^*\mG_1^lX=\mG_2^lX^*X=\mG_2^{l-1}+z\mG_2^{l}. \label{relationXG}
\end{align}
\end{lem}

 Further, for $a \in \lb 1, M\rb$ and $b\in \lb 1, N\rb$, we denote by  $E_{ab}$ the $M\times N$ matrix with entires $(E_{ab})_{cd}=\delta_{ac}\delta_{bd}$. Let
\begin{align}
\mathscr{P}_{0}^{ab}=E_{ab}(E_{ab})^*,\quad \mathscr{P}_{1}^{ab}=E_{ab}X^*,\quad \mathscr{P}_{2}^{ab}=X(E_{ab})^*. \label{P012}
\end{align}
For any integer $l\geq1$, it is also elementary to compute that 
\begin{align}
&\frac{\partial \mG_1^{l}}{\partial x_{ab}}=-\sum_{\alpha=1}^2\sum_{\begin{subarray}{c}l_1,l_2\geq 1\\ l_1+l_2=l+1 \end{subarray}}\mG_1^{l_1}\mathscr{P}_{\alpha}^{ab}\mG_1^{l_2}.\label{derivative}
\end{align}
Repeatedly applying the  identity \rf{derivative}, we can  get the formulas for higher order derivatives of $\mG_1^l$ w.r.t.  $x_{ab}$. Moreover, by (\ref{derivative}) and  the product rule, we can easily deduce the derivatives of $X^*\mG_1^{l}$ w.r.t. $x_{ab}$.   For the convenience of the reader, we collect more basic formulas of the  derivatives of Green functions in Appendix \ref{s.derivative of G}.

\section{Green function representation}\label{secgfr}
In this section, we  express $\mu_i$ and $ |\langle \mb{w}, \mb{\xi}_i\rangle|^2$ in terms of  the Green function $\mG_1(z)$ in \eqref{def:green}. This representation will allow us to work with the Green function instead of the eigenvalue and eigenvector statistics.  We also remark here that similar derivation of the Green function representation has appeared in previous work such as \cite{KY13, bloemendal2016principal}. But here for eigenvectors, we need to do it up to a higher order precision, in order to capture all contributing terms for the fluctuation.  

We start with a few more notations. We define the centered Green function by
\begin{align}
\Xi(z) := \mG_1(z) - m_1(z) I, \label{19071905}
\end{align} 
and introduce its quadratic forms
\begin{align}
\chi_{ij} (z) = \bv_i^* \Xi(z) \bv_j,\qquad \chi_{\bu j}(z) = \bu^* \Xi(z) \bv_j,\qquad i,j\in \lb 1, r\rb, \label{19071920}
\end{align}
where $\mathbf{u}$ is defined in (\ref{19071820}). For brevity, we further set
 \begin{align}\label{def:wtw}
\widetilde{\mb{w}}:= \Sigma^{-\frac{1}{2}}\mb{w}=\sum_{j=1}^r \wt w_j\bv_j+\bu \quad\text{with}\quad \wt w_j:= \frac{\langle \mb{w},\bv_j \rangle}{\sqrt{1+d_j}}.
 \end{align}
Also, for $d>0$, we define the following functions
\begin{align}
&f(d):=\frac{1}{d}(d+1)(d^2-y), \qquad g(d):=\frac{1}{d} (d+1)(d+y)(d^2-y), \label{19071921}
\end{align}
and for $i\in \lb 1,r\rb$, we set for $d\neq d_i$, 
\begin{align}
&\nu_i (d) := \frac{d_i (d+1)}{d_i -d}. \label{19071945}
\end{align}

With the above notations, we have the following lemma. 
\begin{lem} \label{representation}
Under Assumptions \ref{assumption} , \ref{supercritical}, and the setting (\ref{19071802}),  for $i \in \lb 1, r_0\rb$, we have
\begin{align}
\mu_i&=\theta(d_i)-(d_i^2-y)\theta(d_i)\chi_{ii}(\theta(d_i))+O_\prec\Big(\frac 1N\Big), \label{decompeigenv}
\end{align}
and
\begin{align}
 |\langle \mb{w}, \mb{\xi}_i\rangle|^2&= \frac{d_i^2-y}{d_i (d_i+y)}  |\langle \mb{w}, \bv_i\rangle|^2 - 2 d_i (d_i+1)^2 \wt w_i^2 \chi_{ii}(\theta(d_i)) -f(d_i)^2 \wt w_i^2 \chi_{ii}'(\theta(d_i)) \nonumber\\
 &\qquad - 2f(d_i) \wt w_i\Big(\sum_{j\neq i} \nu_i(d_j)  \wt w_j \chi_{ij}(\theta(d_i))+\chi_{\bu i}(\theta(d_i))\Big)\nonumber\\
 & \qquad+ g(d_i)\Big(\sum_{j\neq i} \nu_i(d_j)  \wt w_j \chi_{ij}(\theta(d_i))+\chi_{\bu i}(\theta(d_i))\Big)^2 + O_\prec\left( \frac{\wt w_i}{N} \right)+O_\prec(N^{-\frac32}). \label{rankrdecomp}
 \end{align}
\end{lem}
\begin{rmk}\label{rmk:representation} Lemma \ref{representation} suggests that  the joint distribution of $\mu_i$ and $ |\langle \mb{w}, \mb{\xi}_i\rangle|^2$ is ultimately governed by the joint distribution of $\chi_{ij}(z)$ ($1\le j \le r$), $\chi_{ii}'(z)$ and $\chi_{\bu i}(z)$. We can also rewrite  \eqref{rankrdecomp} as
\begin{align}
 |\langle \mb{w}, \mb{\xi}_i\rangle|^2&= \frac{d_i^2-y}{d_i (d_i+y)}  |\langle \mb{w}, \bv_i\rangle|^2 - \frac{\wt{w}_i}{\sqrt{N}}{\mathbf{l}_i}^*{\mathbf{\chi}}_i + \frac{1}{N}{{\mathbf{\chi}}_i}^*A_i {\mathbf{\chi}}_i+ O_\prec\left( \frac{\wt w_i}{N} \right)+O_\prec(N^{-\frac32}), \label{rankrdecomp1}
\end{align}
by defining the  random vector 
\begin{align*}
{\mathbf{\chi}}_i \equiv {\mathbf{\chi}}_i(\theta(d_i)):= \sqrt N
\begin{pmatrix}
\chi_{i1}(\theta(d_i)),&\cdots,&\chi_{ir}(\theta(d_i)),& \chi_{\bu i}(\theta(d_i)),& \chi_{ii}'(\theta(d_i)) 
\end{pmatrix}^*,
\end{align*}
the deterministic vector $\mathbf{l}_i=(l_{i}(j))\in \mathbb{R}^{r+2}$ with components
$$l_{i}(j) =
\begin{cases}
  2 d_i (d_i+1)^2\wt w_i, & \text{if } j=i;\\
  2f(d_i) \nu_i(d_j)  \wt w_j, & \text{if } 1\le j \neq i \le r;\\
 2f(d_i) , & \text{if } j=r+1;\\
  f(d_i)^2 \wt w_i, & \text{if } j=r+2,
\end{cases}
$$
and the $(r+2)\times (r+2)$ symmetric matrix $A_i$ whose  non-zero entries are given by
$$A_i (j,k)=
\begin{cases}
 g(d_i) \nu_i(d_j) \nu_i(d_k) \wt w_j \wt w_k, & \text{if } 1\le j,k\le r \text{ and } j,k \neq i;\\
  g(d_i) \nu_i(d_j) \wt w_j, & \text{if } 1\le j\neq i\le r \text{ and } k=r+1;\\
g(d_i),& \text{if } j=k=r+1.
\end{cases}
$$
\end{rmk}

 \begin{proof}[Proof of Lemma \ref{representation}]
 First, the proof of \rf{decompeigenv} can be done similarly to the proof of Proposition $7.1$ in \cite{KY13}. For the convenience of the reader, we state the details below.  Recall (\ref{19071901}) together with (\ref{19071902}).  We rewrite   $S$ as
$$S=V\diag(d_1,\cdots,d_r)V^*,$$ 
by setting $V=(\bv_1,\cdots,\bv_r)$. Therefore, we have 
 $$\Sigma^{-1}=I-VDV^*.$$
 with
 $$D=\diag\left(\frac{d_1}{1+d_1},\cdots,\frac{d_r}{1+d_r} \right).$$
  Then, by elementary calculation, we have 
 \begin{align*}
 Q-z=\Sigma^{1/2}\mG_1^{-1}(z)\big(I_M+z\mG_1(z)VDV^*\big)\Sigma^{1/2}.
 \end{align*}
 Notice that $\mu_i$ is the $i$th largest real value such that $\det{(Q-\mu_i)}=0$. Further by the fact that $\mu_i$ stays away from the spectrum of $H$ with high probability (\cf Lemma  \ref{locationeig},  Assumption \ref{supercritical}), together with the identity $\det \big(I_M+z\mG_1(z)VDV^*\big)=\det(D)\det  \big(D^{-1}+zV^*\mG_1(z)V\big)$, we have that $\mu_i$ is the $i$th largest real solution to the equation $\det  \big(D^{-1}+z V^*\mG_1(z)V\big)=0$ with high probability. 
 
For $x\in [\lambda_++\delta, K]$ with sufficiently small constant $\delta>0$ and sufficiently large constant $K>0$,  we define the matrices $\mathcal{A}(x)=(\mathcal{A}_{ij}(x))$ and $\wt{\mathcal{A}}(x)=(\wt{\mathcal{A}}_{ij}(x))$ by setting
 \begin{align*}
 \mathcal{A}_{ij}(x)=(1+d_i^{-1})\delta_{ij}+x\mb{v}_i^*\mG_1(x)\mb{v}_j-xm_1(x)\delta_{ij}, \quad \wt{\mathcal{A}}_{ij}(x)=\delta_{ij}((1+d_i^{-1})+x\mathbf{v}_i^*\mG_1(x)\mb{v}_i-xm_1(x)),
 \end{align*}
 Further, we denote the eigenvalues of $\mathcal{A}(x)$ and $\wt{\mathcal{A}}(x)$ by $a_1(x)\leq \ldots\leq a_r(x)$ and $\wt{a}_1(x)\leq \ldots\leq \wt{a}_r(x)$ respectively. Apparently, one has
 \begin{align}
 \wt{a}_i(x)= (1+d_i^{-1})+x\mathbf{v}_i^*\mG_1(x)\mb{v}_i-xm_1(x), \label{19072601}
 \end{align}
 with high probability by the isotropic local law (\ref{est_DG}) and the Assumption \ref{supercritical}. We then claim that, in order to prove (\ref{decompeigenv}), it suffices to show  the following two estimates
 \begin{align}
 &\mu_im_1(\mu_i)=-a_i(\theta(d_i))+O_\prec(\frac{1}{N}),\label{19072602}\\
 & a_i(\theta(d_i))=\wt{a}_i(\theta(d_i))+O_\prec(\frac{1}{N}).  \label{19072603}
 \end{align}
 Combining (\ref{19072601}), (\ref{19072602}) and (\ref{19072603}), we have 
 \begin{align*}
\mu_im_1(\mu_i)- \theta(d_i)m_1(\theta(d_i))= -(1+d_i^{-1})-\theta(d_i)\mathbf{v}_i^*\mG_1(\theta(d_i))\mb{v}_i. 
 \end{align*}
 Expanding $\mu_im_1(\mu_i)$ around $\theta(d_i)m_1(\theta(d_i))$ with the aid of Lemma \ref{locationeig}  will then lead to  (\ref{decompeigenv}). Therefore, what remains is to prove (\ref{19072602}) and (\ref{19072603}).  
 
 We start with (\ref{19072602}).  First, by the fact that $\mu_i$ is a solution to $\det  \big(D^{-1}+z V^*\mG_1(z)V\big)=0$, it is easy to see that 
 $\mu_im_1(\mu_i)= -a_k(\mu_i)$ for some $k$.  But by isotropic law (\ref{est_DG}) and Lemma \ref{locationeig}, we see that  $\mathcal{A}=\wt{\mathcal{A}}+O_\prec(N^{-\frac12})$, where $O_\prec(N^{-\frac12})$ represents a matrix bounded in  operator norm  by $O_\prec(N^{-\frac12})$. This leads to the  estimate $a_k(\mu_i)=\wt{a}_k(\mu_i)+O_\prec(N^{-\frac12})$. Further, by  (\ref{est_DG}) and Lemma \ref{locationeig} one can easily show that $\mu_im_1(\mu_i)=-(1+d_i^{-1})+O_\prec(N^{-\frac12})$ and $\wt{a}_k(\mu_i)=(1+d_k^{-1})+O_\prec(N^{-\frac12})$.  Therefore, due to the fact that $d_i$'s are well separated, we have $\mu_im_1(\mu_i)= -a_i(\mu_i)$  with high probability. Next, by the isotropic law (\ref{est_DG}), one can also check that 
 \begin{align*}
 \|\partial_x \mathcal{A}(x)\|_{\text{op}}\prec\frac{1}{\sqrt{N}}. 
 \end{align*}
 This together with Lemma \ref{locationeig} leads to
 \begin{align*}
 |a_i(\mu_i)-a_i(\theta(d_i))|\prec \frac{1}{N}.
 \end{align*}
 Combining the above with the fact $\mu_im_1(\mu_i)= -a_i(\mu_i)$  with high probability, we arrive at (\ref{19072602}). 
 
 Next, we prove (\ref{19072603}).  Observe that the diagonal entries of $\mathcal{A}-\wt{\mathcal{A}}$ are $0$, and $\wt{\mathcal{A}}$ is a diagonal matrix. So expanding the eigenvalues of $\mathcal{A}$ around the eigenvalues of $\wt{\mathcal{A}}$ using the perturbation theory, we see that the first order  term vanishes. Hence, it suffices to estimate the second order term. More specifically, we have 
 \begin{align*}
 |a_i(\theta(d_i))-\wt{a}_i(\theta(d_i))|\prec \frac{\|\mathcal{A}-\wt{\mathcal{A}}\|_{\text{op}}^2}{\min_{j\neq i} |\wt{a}_j(\theta(d_i))-\wt{a}_i(\theta(d_i))|}\prec\frac{1}{N},
 \end{align*} 
 where the last step follows from the fact that $\wt{a}_k(\theta(d_i))=1+d_k^{-1}+O_\prec(N^{-\frac12})$ and the fact that the $d_i$'s are well separated. This concludes the proof of (\ref{19072603}).

Then,  we turn to  prove \rf{rankrdecomp}.
 Let $\Gamma_i$ be the boundary of a disc centered at $d_i$ with sufficiently  small (but fixed) radius,  such that this disc is away from $\lambda_+$ by a constant order distance and also $\Gamma_i$'s are well-separated from each other by a constant order distance for all $i\in \lb 1, r_0\rb$. Notice that this can be guaranteed by Assumption \ref{supercritical}. According to Lemma  \ref{locationeig}, together with the Cauchy integral, we have the following equality with high probability 
\begin{align}
|\langle\mb{w},\mb{\xi}_i\rangle|^2=-\frac{1}{2\pi \ii}\oint_{\theta(\Gamma_i)}\mb{w}^*G(z)\mb{w}\,{\rm d}z,\label{greenfunctionrepre}
\end{align}
 where $\theta(\Gamma_i)$ is the image of $\Gamma_i$ under the map $\theta(\cdot)$ defined in (\ref{def. of theta}). 
 With the notations for $V, S, D$ and $\Sigma^{-1}$, using the setting (\ref{19071802}), we can write 
\begin{align*}
G(z)&=\left( \Sigma^{\frac12} XX^*  \Sigma^{\frac12} - z I \right)^{-1}= \Sigma^{-\frac12} \left( \mG_1^{-1}(z) +z \Sigma^{-\frac12} S \Sigma^{-\frac12} \right)^{-1} \Sigma^{-\frac12}\\
& = \Sigma^{-\frac12} \left( \mG_1^{-1}(z) + z VDV^* \right)^{-1} \Sigma^{-\frac12}.
\end{align*}
Then, it follows from the matrix inversion lemma that
\begin{align*}
G(z)&= \Sigma^{-\frac12}\mG_1(z) \Sigma^{-\frac12}-z \Sigma^{-\frac12}\mG_1(z) V\big(D^{-1}+zV^*\mG_1(z) V\big)^{-1} V^*\mG_1(z) \Sigma^{-\frac12}. 
\end{align*}
With the notation introduced in \eqref{def:wtw}, we can further write
 \begin{align}
 \mb{w}^*G(z)\mb{w}=\widetilde{\mb{w}}^*\mG_1(z)\widetilde{\mb{w}}-z\widetilde{\mb{w}}^*\mG_1(z){V}\big(D^{-1}+zV^*\mG_1(z)V\big)^{-1}V^*\mG_1(z) \widetilde{\mb{w}}. \label{19071903}
 \end{align}
Plugging (\ref{19071903}) into \eqref{greenfunctionrepre}, and noticing that the contour integral of $\widetilde{\mb{w}}^*\mG_1(z)\widetilde{\mb{w}}$ on $\theta(\Gamma_i)$ is zero with high probability by Assumption \ref{supercritical} and the rigidity of eigenvalues of $H$ (\cf (\ref{190726100})), one has
  \begin{align}\label{eq:intrep1}
|\langle\mb{w},\mb{\xi}_i\rangle|^2=\frac{1}{2\pi \ii}\oint_{\theta(\Gamma_i)} z\widetilde{\mb{w}}^*\mG_1(z){V}\big(D^{-1}+zV^*\mG_1(z)V\big)^{-1}V^*\mG_1(z) \widetilde{\mb{w}}\,{\rm d}z,
  \end{align}
with high probability. 

For the integrand in (\ref{eq:intrep1}), we first recall the notation in (\ref{19071905}) and then 
we apply resolvent expansion 
\begin{align}
\big(D^{-1}+zV^*\mG_1(z)V\big)^{-1} 
=&L(z) - z L(z) V^*\Xi(z)V L(z) \nonumber\\
&+ \big( z L(z)V^*\Xi(z)V \big)^2 \big(D^{-1}+zV^*\mG_1(z)V\big)^{-1}, \label{19071907}
\end{align}
where 
$$L(z) := \left( D^{-1} + z m_1(z)  \right)^{-1}.$$
With (\ref{19071907}), we can further rewrite  \eqref{eq:intrep1}  as
\begin{align*}
|\langle\mb{w},\mb{\xi}_i\rangle|^2=\frac{1}{2\pi \ii} &\oint_{\theta(\Gamma_i)} z \big( m_1(z) \wt{\bw}^*V +\wt{\bw}^* \Xi(z) V \big)\Big( L(z) - z L(z) V^*\Xi(z)V L(z)\\
& + \big( z L(z)V^*\Xi(z)V \big)^2 \big(D^{-1}+zV^*\mG_1(z)V\big)^{-1} \Big) \big( m_1(z) V^* \wt{\bw} + V^* \Xi(z) \wt{\bw} \big) \,{\rm d}z.
\end{align*}
Applying \eqref{est_DG}, we  can further write
\begin{align*}
|\langle\mb{w},\mb{\xi}_i\rangle|^2 = S_1 + S_2 + S_3 + O_\prec\big(N^{-\frac{3}{2}}\big),
\end{align*}
by defining
\begin{align*}
&S_1:= \frac{1}{2\pi \ii}\oint_{\theta(\Gamma_i)} z m_1^2(z) \wt{\bw}^* V L(z) V^* \wt\bw \,{\rm d}z,\\
&S_2:= \frac{1}{2\pi \ii}\oint_{\theta(\Gamma_i)} \Big( 2 z m_1(z)\wt{\bw}^* V L(z) V^* \Xi(z) \wt\bw - z^2 m_1^2(z) \wt{\bw}^* V L(z) V^* \Xi(z) V L(z) V^* \wt\bw   \Big) \,{\rm d}z,\\
&S_3:=  \frac{1}{2\pi \ii} \oint_{\theta(\Gamma_i)} \Big( z  \wt\bw^* \Xi(z) V L(z) V^* \Xi(z) \wt\bw - 2 z^2 m_1(z) \wt\bw^* (VL(z)V^* \Xi(z))^2 \wt\bw\\
&\qquad \qquad \qquad \qquad+ z^3 m_1^2(z) \wt\bw^* (VL(z)V^* \Xi(z))^2V L(z) V^* \wt\bw  \Big) \,{\rm d}z.
\end{align*}

It remains to estimate $S_1,S_2$ and $S_3$. From the definitions in (\ref{m1m2}) and (\ref{def. of theta}),  
it is easy to check the identity
\begin{align}
1+z^{-1} + \theta(z) m_1 (\theta(z))=0.  \label{19071910}
\end{align} 
With the above identity, we see that
\begin{align*}
V L(\theta(z)) V^* = \sum_{j=1}^r  \frac{\bv_j \bv_j^*}{1+d_j^{-1} + \theta(z) m_1 (\theta(z))} = \sum_{j=1}^r  \frac{ z d_j \bv_j \bv_j^*}{z- d_j}.
\end{align*}
Therefore, by the residue theorem,
\begin{align}
S_1 &= \frac{1}{2\pi \ii}\oint_{\Gamma_i}  \theta(z) \theta'(z) m_1^2(\theta(z)) \wt{\bw}^* V L(\theta(z)) V^* \wt\bw \,{\rm d}z \nonumber\\
&= \theta(d_i) \theta'(d_i) m_1^2 (\theta(d_i)) d_i^2 (\wt\bw^* \bv_i)^2= \frac{d_i^2 - y }{d_i (d_i+ y)} |\langle \bw,\bv_i \rangle|^2. \label{19071940}
\end{align}
Similarly,  using (\ref{19071910}), we can get 
\begin{align*}
S_2 &= \frac{1}{2\pi \ii}\oint_{\Gamma_i}  \Big( 2\theta(z) \theta'(z) m_1(\theta(z))  \sum_{j=1}^r \frac{z d_j}{z- d_j} (\wt\bw^* \bv_j) \wt\bw^* \Xi(\theta(z)) \bv_j \\
&\quad- \theta^2(z) \theta'(z) m_1^2(\theta(z)) \sum_{j,k=1}^r \frac{z^2 d_j d_k}{(z-d_j)(z-d_k)} (\wt\bw^* \bv_j) (\wt\bw^* \bv_k) \bv_j^*\Xi(\theta(z)) \bv_k \Big) \,{\rm d}z.
\end{align*} 
Further by the residue theorem together with the definition of $\theta$ and $m$ in (\ref{m1m2}) and (\ref{def. of theta}), we can get
\begin{align*}
S_2 &= -2\frac{(d_i+1)(d_i^2-y)}{d_i} {\wt w_i}\wt\bw^*\Xi(\theta(d_i))\bv_i -2 \frac{(d_i+1)^2(d_i^2-y)}{d_i} \sum_{j\neq i} \frac{d_j}{d_i-d_j} {\wt w_i} {\wt w_j} \chi_{ij}(\theta(d_i))\\
&-\frac{(d_i+1)^2(d_i^2-y)^2}{d_i^2} {\wt w_i}^2 \chi_{ii}'(\theta(d_i)) - 2\frac{(d_i+1)(d_i^3+y)}{d_i} {\wt w_i}^2 \chi_{ii}(\theta(d_i)),
\end{align*}
where we also recalled the notations in (\ref{19071920}). With the functions defined in (\ref{19071921}) and (\ref{19071945}), we can further write 
\begin{align}
S_2=&-2d_i (d_i+1)^2 {\wt w_i}^2 \chi_{ii}(\theta(d_i)) - 2 f(d_i) \sum_{j\neq i} \nu_i(d_j) \wt w_i \wt w_j \chi_{ij}(\theta(d_i)) \nonumber\\
&- 2 f(d_i) \wt w_i \chi_{\bu i}(\theta(d_i)) - f(d_i)^2 \wt w_i^2 \chi_{ii}'(\theta(d_i)).  \label{19071941}
\end{align}

Observe from the isotropic local law \eqref{est_DG} and Assumption \ref{supercritical} that $S_2=O_\prec\big(\wt{w}_i/\sqrt{N}\big)$. Hence,  it can degenerate when $\wt{w}_i$ is small or even $0$. Hence, it is necessary to consider the fluctuation of the  term $S_3$ as well.  In the sequel, we turn to estimate $S_3$.  We estimate the integrals of three terms in the integrand separately. First, using the residue theorem together with the notations in (\ref{19071920}) and (\ref{19071921}), 
we have  
\begin{align}
&\frac{1}{2\pi \ii} \oint_{\theta(\Gamma_i)}z  \wt\bw^* \Xi(z) V L(z) V^* \Xi(z) \wt\bw \,{\rm d}z = \theta(d_i) \theta'(d_i) d_i^2 \left(\wt\bw^*\Xi(\theta(d_i)) \bv_i \right)^2 \nonumber\\
&=g(d_i) \Big(\sum_{j,k=1}^r \wt w_j \wt w_k \chi_{ij}(\theta(d_i))\chi_{ik}(\theta(d_i)) + \chi_{\bu i}^2 (\theta(d_i)) + 2 \sum_{j=1}^r \wt w_j \chi_{ij}(\theta(d_i))\chi_{\bu i} (\theta(d_i)) \Big). \label{19071930}
\end{align}
For the second part of the integral, we have 
\begin{align*}
&-\frac{1}{2\pi \ii} \oint_{\theta(\Gamma_i)}  2 z^2 m_1(z) \wt\bw^* (VL(z)V^* \Xi(z))^2 \wt\bw \,{\rm d}z\\
&=-2 \theta^2(d_i)\theta'(d_i)m_1(\theta(d_i)) \sum_{j\neq i} \frac{d_i^3 d_j}{d_i - d_j} (\wt\bw^*\bv_j)\wt\bw^*\Xi(\theta(d_i))\bv_i\chi_{ij}(\theta(d_i)) \\
&\quad -2 \theta^2(d_i)\theta'(d_i)m_1(\theta(d_i)) \sum_{j\neq i} \frac{d_i^3 d_j}{d_i - d_j} (\wt\bw^*\bv_i)\wt\bw^*\Xi(\theta(d_i))\bv_j\chi_{ij}(\theta(d_i)) \\
&\quad-2d_i^2 (\wt\bw^*\bv_i) \Big(\theta^2(z)\theta'(z)m_1(\theta(z))z^2 \wt\bw^*\Xi(\theta(z))\bv_i \chi_{ii}(\theta(z)) \Big)'\Big|_{z=d_i}.
\end{align*}

Then,  by using the isotropic law \eqref{est_DG}, Assumption \ref{supercritical}, and  the notations in (\ref{19071920}) and (\ref{19071921}),  we have 
\begin{align}
&-\frac{1}{2\pi \ii} \oint_{\theta(\Gamma_i)}  2 z^2 m_1(z) \wt\bw^* (VL(z)V^* \Xi(z))^2 \wt\bw \,{\rm d}z\nonumber\\
&=-2d_i^3 \theta^2(d_i)\theta'(d_i)m_1(\theta(d_i)) \sum_{j\neq i} \frac{ d_j}{d_i - d_j} (\wt\bw^*\bv_j)\wt\bw^*\Xi(\theta(d_i))\bv_i \chi_{ij}(\theta(d_i)) + O_\prec\left( \frac{\wt{w}_i}{N}\right)  \nonumber\\
&=2 g(d_i) \sum_{j\neq i} \frac{(1+d_i)d_j}{d_i-d_j} \wt w_j \chi_{ij} (\theta(d_i)) \Big( \sum_{k\neq i} \wt w_k\chi_{ik} (\theta(d_i)) + \chi_{\bu i} (\theta(d_i))\Big)+ O_\prec\left( \frac{\wt{w}_i}{N}\right). \label{19071931}
\end{align}
Analogously,  the last part of the integral can be estimated by  
\begingroup
\allowdisplaybreaks
\begin{align}
&\frac{1}{2\pi \ii} \oint_{\theta(\Gamma_i)}  z^3 m_1^2(z) \wt\bw^* (VL(z)V^* \Xi(z))^2V L(z) V^* \wt\bw  \,{\rm d}z   \nonumber\\
&=g(d_i) \sum_{k,l\neq i}\frac{(1+d_i)^2 d_k d_l}{(d_i - d_k) (d_i -d_l)} \wt w_k \wt w_l \chi_{ik} (\theta(d_i)) \chi_{il} (\theta(d_i))+ O_\prec\left( \frac{\wt{w}_i}{N}\right). \label{19071932}
\end{align}
\endgroup
Combining (\ref{19071930})-(\ref{19071932}), 
after necessary simplification, we arrive at
\begin{align}
S_3=&g(d_i)\Big( \chi_{\bu i}^2 (\theta(d_i))+ 2 \sum_{j\neq i} \nu_i(d_j) \wt w_j \chi_{ij}(\theta(d_i)) \chi_{\bu i} (\theta(d_i)) \nonumber\\
&+ \sum_{j,k\neq i} \nu_i(d_j) \nu_i(d_k) \wt w_j \wt w_k \chi_{ij}(\theta(d_i))\chi_{ik}(\theta(d_i)) \Big)+ O_\prec\left( \frac{\wt{w}_i}{N}\right). \label{19071942}
\end{align}
By (\ref{19071940}), (\ref{19071941}) and (\ref{19071942}), we can conclude the proof of Lemma \ref{representation}.
 \end{proof}

\section{Proof of Theorem \ref{rankrmainthm} } \label{s.proof of thm}

In this section, we state  the proof of the  main result Theorem \ref{rankrmainthm}, based on Proposition \ref{recursivemP}, whose proof will be stated in Section \ref{sec:prop}. The starting point  is Lemma \ref{representation} and  Remark \ref{rmk:representation}, which state that  the study of $|\langle \bw,\mb{\xi}_i \rangle|^2$ can be  reduced to the study of the random vector 
\begin{align*}
{\mathbf{\chi}}_i(z) := \sqrt N
\begin{pmatrix}
\chi_{i1}(z),&\cdots,&\chi_{ir}(z),& \chi_{\bu i}(z),& \chi_{ii}'(z) 
\end{pmatrix}^T
\end{align*}
at  $z=\theta(d_i)$. It suffices to show that ${\mathbf{\chi}}_i(z)$ is asymptotically real Gaussian. Before we give the precise statement, let us introduce some necessary notations. For brevity, in the sequel, we will very often omit the $z$-dependence from the notations. For instance, we will write $m_{1,2}(z)$ as $m_{1,2}$. 

In the sequel, we fix an $i\in \lb 1, r_0\rb$. Define a symmetric matrix  $\mathcal{M}_i\equiv \mathcal{M}_i(z)\in{\mathbb{C}^{(r+2)\times (r+2)}}$ with diagonal entries 
\begin{align}\label{def:Mdiag}
\mathcal M_i (j,j) = 
\begin{cases}
2m_1^2(zm_1)',&\text{ if } j=i;\\
m_1^2(zm_1)',&\text{ if } 1\le j \neq i\le r;\\
m_1^2(zm_1)'\Vert\bu\Vert^2, &\text{ if } j=r+1;\\
2\big(m_1m_1'(zm_1)''+(m_1')^2(zm_1)'+\frac{1}{6}m_1^2(zm_1)'''\big),&\text{ if } j=r+2
\end{cases}
\end{align}
and the only non-zero off-diagonal entry
\begin{align}\label{def:Moff}
\mathcal M_i (i,r+2) = \big(m_1^2(zm_1)'\big)'.
\end{align}
Recall the notations defined in (\ref{def:s12}).  We then further define the  symmetric matrix  $\mathcal{K}_i\equiv \mathcal{K}_i(z)\in \mathbb{C}^{(r+2)\times (r+2)}$ whose  $r\times r$ upper left corner is given by
\begin{align*}
\mathcal K_i(j,k) = \mathbf{s}_{1,1,2}(\bv_j,\bv_k,\bv_i) (zm_2m_1^2)^2, \quad\text{for } 1\le j,k\le r,
\end{align*}
and the remaining entries are
\begin{align}\label{def:K}
\mathcal K_i(j,k)=
\begin{cases}
\mathbf{s}_{1,1,2} (\bv_j, \bu,\bv_i) (zm_2m_1^2)^2,& \text{if } 1\le j \le r \text{ and } k=r+1;\\
\frac{1}{2} \mathbf{s}_{1,3} (\bv_j,\bv_i) \big((zm_2m_1^2)^2\big)', & \text{if } 1\le j \le r \text{ and } k=r+2;\\
\mathbf{s}_{2,2} (\bu,\bv_i) (zm_2m_1^2)^2,& \text{if } j=k=r+1;\\
\mathbf{s}_4(\bv_i) \big((zm_2m_1^2)'\big)^2,& \text{if } j=k=r+2;\\
\frac{1}{2} \mathbf{s}_{1,3} (\bu,\bv_i) \big((zm_2m_1^2)^2\big)', & \text{if } j=r+1 \text{ and } k=r+2.
\end{cases}
\end{align}

We further denote 
\begin{align}\label{def:variance}
\mathcal{V}_i(z):=\mathcal{M}_i(z)+\kappa_4\mathcal{K}_i(z).
\end{align}

\begin{lem}\label{lem:normal} Under the assumptions of Theorem \ref{rankrmainthm}, we have
$${\mathbf{\chi}}_i(\theta(d_i)) \simeq \mathcal N\left(\mathbf{0}, \mathcal V_i(\theta(d_i)) \right).$$
\end{lem}

With the above lemma, we now can finish the proof of Theorem \ref{rankrmainthm}.
\begin{proof}[Proof of Theorem \ref{rankrmainthm}]
Recall  Lemma \ref{representation}. Then, by setting
\begin{align}
&\Upsilon_i=-\sqrt{N}(d_i^2-y)\theta(d_i)\chi_{ii}(\theta(d_i)), \nonumber\\ 
&\Theta_i^{\mb{w}}=- 2\sqrt{N} d_i (d_i+1)^{3/2} \wt w_i \chi_{ii}(\theta(d_i)) -\sqrt{N}f(d_i)^2 (1+d_i)^{-1/2}\wt w_i \chi_{ii}'(\theta(d_i)) \nonumber\\
 & - 2\sqrt{N}f(d_i)(1+d_i)^{-1/2}\Big(\sum_{j\neq i} \nu_i(d_j)  \wt w_j \chi_{ij}(\theta(d_i))+\chi_{\bu i}(\theta(d_i))\Big),\nonumber\\
&\Lambda_i^{\mb{w}}=\sqrt{N}\sqrt{g(d_i)}\Big(\sum_{j\neq i} \nu_i(d_j)  \wt w_j \chi_{ij}(\theta(d_i))+\chi_{\bu i}(\theta(d_i))\Big),
\end{align}
we get \rf{thmdecompeigenv} and \eqref{rankrdecompv3} in  Theorem \ref{rankrmainthm}. Next, notice that $\Upsilon_i$,  $\Theta_i^{\mb{w}}$ and $\Lambda_i^{\mb{w}}$ are all linear combinations of the components of $\mathbf{\chi}_i(\theta(d_i))$. Then by Lemma \ref{lem:normal}, they are also asymptotically jointly Gaussian with mean $0$. Further,  elementary calculation of the quadratic forms of  the entries of $\mathcal V_i(\theta(d_i))$ using the following basic facts at $z=\theta(d_i)$
\begin{align*}
&m_1^2(zm_1)' = \frac{1}{(d_i+y)^2(d_i^2-y)},\\
&m_1m_1'(zm_1)''+(m_1')^2(zm_1)'+\frac{1}{6}m_1^2(zm_1)'''=\frac{d_i^4(4d_i^4+4d_i^3y-3d_i^2y+d_i^2y^2+y^2+y^3)}{(d_i+y)^4(d_i^2-y)^5},\\
&zm_2m_1^2=-\frac{1}{d_i(d_i+y)},\qquad(zm_2m_1^2)'=\frac{2d_i+y}{(d_i+y)^2(d_i^2-y)},
\end{align*}
eventually leads to the covariance matrix of  $(\Upsilon_i, \Theta_i^{\mb{w}}, \Lambda_i^{\mb{w}})$, which is   stated in  Theorem \ref{rankrmainthm}.

This completes the proof of Theorem \ref{rankrmainthm}.
\end{proof}

In the rest of this section, we prove Lemma \ref{lem:normal}, based on our key technical result, Proposition \ref{recursivemP}. In order to show the asymptotic Gaussianity of $\mathbf{\chi}_i(\theta(d_i))$, it suffices to show that all  linear combinations of the components of $\mathbf{\chi}_i(\theta(d_i))$ are asymptotic Gaussian.  Our proof will be based on a moment estimate. This requires a deterministic bound for the Green function,  in order to control the contribution of the bad event in the isotropic local laws. To this end, we introduce a tiny imaginary part to the parameter $z$, such that the Green functions can be bounded by $1/\Im z$ deterministically. Specifically, in the sequel, we set 
\begin{align}
z=\theta(d_i)+\ii N^{-K}, \label{19072401}
\end{align}
for some sufficiently large constant $K$. For a fixed deterministic $(r+2)$-dim column vector $\mb{c}=(c_1,\,\cdots,\,c_{r+2})^*$, we define
\begin{align}
\mP:=\mb{c}^*\mathbf{\chi}_i(z), \label{19071950}
\end{align}
where $z$ is given in (\ref{19072401}).  Notice that $|\mathcal{P}|\prec 1$ by the isotropic local law. 
 Here we omit the dependence of $\mathcal{P}$ on $\mathbf{c}$ and the index $i$ for simplicity. Hereafter we always assume that $\mathbf{c}$ and $i$ are fixed. 
The following proposition is our main technical task. 

\begin{prop}[Recursive moment estimate] \label{recursivemP} Let $\mathcal{P}$ be defined in (\ref{19071950}) with $z$ given in (\ref{19072401}). Under the assumption of  Theorem \ref{rankrmainthm}, we have 
\begin{align}
&(i) \quad\mbE \mathcal{P}=O_\prec(N^{-\frac12}), \label{est_EmP}\\
&(ii)\quad \mbE \mathcal{P}^l =(l-1)\mathcal{V}_{i,\mathbf{c}}\mbE P^{l-2}+O_\prec(N^{-\frac12}),\label{est_EmPk}
\end{align}
where {$\mathcal{V}_{i,\mathbf{c}}=\mb{c}^*{\mathcal{V}_i(\theta(d_i))}\mb{c }$} with  $\mathcal{V}_i(\theta(d_i))$ defined in \eqref{def:variance}.
\end{prop}

With the above proposition, we can now show the proof of Lemma \ref{lem:normal}. 

\begin{proof}[Proof of Lemma \ref{lem:normal}]
By Proposition \ref{recursivemP}, one observes that $\mP(z)$ is asymptotically Gaussian with mean $0$ and variance $\mathcal{V}_{i,\mathbf{c}}$. By the definition of $z$ in (\ref{19072401}), and a simple continuity argument for the Green function, one can easily see that $\mathcal{P}(\theta(d_i))$ admits the same asymptotic distribution as  $\mP(z)$, when $K$ is chosen to be sufficiently large. 
Further implied by the fact that  {$\mathcal{V}_{i,\mathbf{c}}=\mb{c}^*{\mathcal{V}_i(\theta(d_i))}\mb{c }$} and the arbitrariness of $\mb{c}$, we have 
$${\mathbf{\chi}}_i(\theta(d_i)) \simeq \mathcal N\left(\mathbf{0}, \mathcal V_i(\theta(d_i)) \right).$$
This concludes the proof of Lemma \ref{lem:normal}. 
\end{proof}

\section{Proof of Proposition \ref{recursivemP}}\label{sec:prop}
This section is devoted to the proof of Proposition \ref{recursivemP}, the recursive moment estimates of $\mathcal{P}$ defined in (\ref{19071950}). The basic strategy is to use the cumulant expansion formula in Lemma \ref{cumulantexpansion} to the functionals of Green functions. In the context of Random Matrix Theory, such an idea dates back to \cite{KKP96}. We also refer to \cite{LS18, HK} for some recent applications of this strategy for other problems in Random Matrix Theory.

First,  we will see that all the random terms we will encounter in the proof  are  one of the following forms: $\mb{\eta}_1^*X^*\mG_1^s\mb{\eta}_2, \mb{\eta}_1^*X^*\mG_1^sX\mb{\eta}_2, \mb{\eta}_1^*\mG_1^s\mb{\eta}_2$, for some fixed $s\in \mathbb{N}$ and some deterministic vectors $\mb{\eta}_1, \mb{\eta}_2$ which are bounded by some constant $C>0$ in $\ell^2$-norm. Notice that under the choice of $z$ in (\ref{19072401}),
we have the deterministic bound
 \begin{align}\label{trivialbd_G1}
 |\mb{\eta}_1^*\mG_1^s\mb{\eta}_2|\leq C(\Im z)^{-s}.
 \end{align}
Similarly, by Cauchy-Schwarz inequality, we have 
 \begin{align}\label{trivialbd_XG1}
 |\mb{\eta}_1^*X^*\mG_1^s\mb{\eta}_2|&\leq  \|\mb{\eta}_1\|\|X^*\mG_1^s\mb{\eta}_2\|\leq C  \Big(\mb{\eta}_2^*\mG_1^s(\bar{z})XX^*\mG_1^s(z)\mb{\eta}_2 \Big)^{\frac12}\nonumber\\
 &= \Big(\mb{\eta}_2^*\mG_1^s(\bar{z})(I+z\mG_1)\mG_1^{s-1}(z)\mb{\eta}_2 \Big)^{\frac12}\leq C(\Im z)^{-s}, 
  \end{align}
 and 
  \begin{align}\label{trivialbd_XG1X}
 |\mb{\eta}_1^*X^*\mG_1^sX\mb{\eta}_2|= |\mb{\eta}_1^*X^*X\mG_2^s\mb{\eta}_2|= |\mb{\eta}_1^*\mG_2^{s-1}\mb{\eta}_2+z\mb{\eta}_1^*\mG_2^s\mb{\eta}_2|\leq C(\Im z)^{-s}.
 \end{align}
 The above deterministic bounds allow us to use the high probability bounds  for the aforementioned quantities (following from the isotropic local law) directly in the calculation of the expectations in (\ref{est_EmP}) and (\ref{est_EmPk}).  

The main tool for the proof is the cumulant expansion formula in Lemma \ref{cumulantexpansion}. By introducing the following two column vectors
 \begin{align}
 \mb{y}_1=(y_{1a})_{a=1}^M:=\sum_{j=1}^rc_j\bv_j+c_{r+1}\bu,\quad \mb{y}_2=(y_{2a})_{a=1}^M:=c_{r+2}\bv_i, \label{defy12}
 \end{align}
we can rewrite
\begin{align}
\mP=\sn\mb{y}_1^*\big(\mG_1-m_1\big)\bv_i+\sn \mb{y}_2^*\big(\mG_1^2-m_1'\big)\bv_i. \label{rankrrepresentmP}
\end{align}
Hence, we can write 
\begin{align}
\mathbb{E}(\mP^l)
=\sn \mathbb{E}\sum_{t=1}^2\mb{y}_t^*\Big(\mG_1^t-\frac{m_1^{(t-1)}}{(t-1)!}\Big)\bv_i\mP^{l-1}. \label{19071960}
\end{align}
Using the identity
\begin{align*}
\mG_1^t=z^{-1}(H\mG_1^t-\mG_1^{t-1}), \qquad t=1,2, 
\end{align*}
we  rewrite  (\ref{19071960}) as
\begin{align*}
\mathbb{E}(\mP^l)=\sn \mathbb{E}\sum_{t=1}^2\mb{y}_t^*\Big(\frac{1}{(1+m_2)z}H\mG_1^{t}+\frac{m_2}{1+m_2}\mG_1^t-\frac{1}{(1+m_2)z}\mG_1^{t-1}-\frac{m_1^{(t-1)}}{(t-1)!}\Big)\bv_i\mP^{l-1},
\end{align*}
Using the first identity in \rf{identitym1m2}, we further have
\begin{align}
\mathbb{E}(\mP^l)=-m_1 \mathbb{E}\sum_{t=1}^2 \sn \mb{y}_t^*\Big( H\mG_1^{t}+z m_2 \mG_1^t- \mG_1^{t-1}+\frac{m_1^{(t-1)}}{(t-1)! m_1}\Big)\bv_i\mP^{l-1}.\label{rankrEmP^l}
\end{align}

Now, we apply the cumulant expansion formula  to the terms $\sn \mathbb{E}\mb{y}_t^*H\mG_1^{t}\bv_i\mP^{l-1}$ for $t=1,2$. For simplicity, we use the following shorthand notation for the summation
$$\sum_{q,k}=\sum_{q=1}^M \sum_{k=1}^N,$$
and similar shorthand notations are used for single index sums. 

By Lemma \ref{cumulantexpansion},  we have 
\begin{align}
&\sn \mathbb{E}\mb{y}_t^*H\mG_1^{s}\bv_i\mP^{l-1}=\sn \mathbb{E}\sum_{q,k}y_{tq}x_{qk}(X^*\mG_1^s\bv_i)_{k }\mP^{l-1}\nonumber\\
&\qquad \qquad=\sn \mathbb{E}\sum_{q,k}y_{tq}\sum_{\alpha=1}^3\frac{\kappa_{\alpha+1}(x_{qk})}{\alpha!}\frac{\partial^{\alpha}}{\partial x_{qk}^{\alpha}}\Big((X^*\mG_1^s\bv_i)_{k}\mP^{l-1}\Big)+\mathcal{R}_{t,s}, \label{19071971}
\end{align}
where $\mathcal{R}_{t,s}$ satisfies

\begin{align}
|\mathcal{R}_{t,s}|\leq &\sn \sum_{q,k} \bigg(C_\ell\mathbb{E}(|x_{qk}|^{5}) \mathbb{E}\Big(\sup_{|x_{qk}|\leq c}\Big|y_{tq}  \frac{\partial^{4}}{\partial x_{qk}^{4}}\Big((X^*\mG_1^s\bv_i)_{k}\mP^{l-1}\Big)\Big| \Big) \nonumber\\
&+C_\ell\mathbb{E}\big(|x_{qk}|^{5}\mathbbm{1}(|x_{qk}|>c)\big) \mathbb{E}\Big(\sup_{x_{qk}\in \mathbb{R}}\Big|y_{tq}\frac{\partial^{4}}{\partial x_{qk}^{4}}\Big((X^*\mG_1^s\bv_i)_{k}\mP^{l-1}\Big)\Big|\Big) \bigg) \nonumber
\end{align}
for any  $c>0$ and any $C_\ell$ satisfying $C_\ell\leq (C\ell)^{\ell}/\ell!$ with some positive constant $C$.

By the product rule, we have 
\begin{align}
\frac{\partial^{\alpha}}{\partial x_{qk}^{\alpha}}\Big((X^*\mG_1^s\bv_i)_{k}\mP^{l-1}\Big)= \sum_{\begin{subarray}{c}\alpha_1,\alpha_2\geq 0\\\alpha_1+\alpha_2=\alpha\end{subarray}} {\alpha \choose \alpha_1}\frac{\partial^{\alpha_1}(X^*\mG_1^s\bv_i)_{k}}{\partial x_{qk}^{\alpha_1}}\frac{\partial^{\alpha_2}\mP^{l-1}}{\partial x_{qk}^{\alpha_2}}. \label{19071970}
\end{align}

In the sequel, for brevity, we set the notation
 \begin{align}
h_{t,s}(\alpha_1,\alpha_2):=\sn  \sum_{q,k}y_{tq}\frac{\kappa_{\alpha_1+\alpha_2+1}(x_{qk})}{(\alpha_1+\alpha_2)!}\binom{\alpha_1+\alpha_2}{\alpha_1}\frac{\partial^{\alpha_1}(X^*\mG_1^s\bv_i)_{k}}{\partial x_{qk}^{\alpha_1}}\frac{\partial^{\alpha_2}\mP^{l-1}}{\partial x_{qk}^{\alpha_2}}, \quad t,s=1,2. \label{rankrdefhts}
\end{align}
Note that $h_{t,s}(\alpha_1,\alpha_2)$  depends on   $l$ and $i$. However, we drop this dependence for brevity.

Using (\ref{19071970}) and the notation (\ref{rankrdefhts}) to (\ref{19071971}), we can now write
\begin{align}
\sn \mathbb{E}\mb{y}_t^*H\mG_1^{s}\bv_i\mP^{l-1}
&=\sum_{\begin{subarray}{c}\alpha_1,\alpha_2\geq 0\\1\leq \alpha_1+\alpha_2\leq 3\end{subarray}}\mathbb{E}h_{t,s}(\alpha_1,\alpha_2)+\mathcal{R}_{t,s} \label{rankrcumuexpansion^t}.
\end{align} 

In the sequel, we estimate $h_{t,s}(\alpha_1,\alpha_2)$  and the remainder terms $\mathcal R_{t,s}$ for $s, t=1,2$.   We collect the estimates  in the following lemma,  whose proof will be postponed to the end of this section.
\begin{lem}\label{rankrlemh_ts} Let $l$ be any fixed positive integer. 
With the convention $m_a^{(-1)}/(-1)!=1$ for $a=1,2$, we have the following estimates on $h_{t,s}(\alpha_1,\alpha_2)$  and $\mathcal{R}_{t,s}$ where $t,s=1,2$. 
\begin{itemize}
\item [(1):] For $h_{t,s}(\alpha_1,\alpha_2)$, the nonnegligible terms are
\begin{align}
&h_{t,s}(1,0)=-\sn\Big(z m_2 \mb{y}_t^*\mG_1^s\bv_i+\sum_{\begin{subarray}{c}1\le s_1,s_2\le s-1 \\s_1+s_2=s\end{subarray}} \frac{(zm_2)^{(s_1)}}{s_1!} \mb{y}_t^*\mG_1^{s_2}\bv_i\Big)\mP^{l-1}+O_\prec(N^{-\frac12}), \label{rankresthts(1,0)}\\
&h_{t,s}(0,1)=-(l-1)\sum_{b=1}^2(\mb{y}_b^*\mb{y}_t+\mb{y}_t^*\bv_i\mb{y}_b^*\bv_i)\sum_{\begin{subarray}{c} 0\le b_1,b_2 \le b-1\\b_1+b_2=b-1\end{subarray}}\frac{m_1^{(b_1)} (zm_1)^{(s+b_2)} }{b_1! (s+b_2)!}\mP^{l-2}+O_\prec(N^{-\frac12}) \label{rankresthts(0,1)},\\
&h_{t,s}(1,2)=-(l-1)\kappa_4\frac{(zm_2m_1)^{(s-1)}}{(s-1)!}
\sum_{b=0}^1 \mathbf{s}_{1,1,2} (\mb{y}_t,\mb{y}_{b+1},\bv_i)(zm_2m_1^2)^{(b)} \mP^{l-2}+O_\prec(N^{-\frac12}) \label{rankresthts(1,2)}.
\end{align}

\vspace{1ex}
\item[(2):]
Except for the terms in (\ref{rankresthts(1,0)})-(\ref{rankresthts(1,2)}), all the other  $h_{t,s}(\alpha_1, \alpha_2)$ terms with $\alpha_1+\alpha_2\le 3$ can be bounded by $O_\prec(N^{-\frac12})$.

\vspace{1ex}
\item [(3):] For the remainder terms, we have 
\begin{align}
\mathcal{R}_{t,s}=\onh. \label{19071980}
\end{align} 
\end{itemize}
\end{lem}

Now we show the proof of Proposition \ref{recursivemP},  based on Lemma \ref{rankrlemh_ts}.
\begin{proof}[Proof of Proposition \ref{recursivemP}]
First we show the proof of \eqref{est_EmP}.
Using  Lemma \ref{rankrlemh_ts}  with $l=1$, we can rewrite \eqref{rankrcumuexpansion^t} as
\begin{align}
\sn \mathbb{E}\mb{y}_t^*H\mG_1^{t}\bv_i=\mathbb{E}h_{t,t}(1,0)+O_\prec(N^{-\frac12}). \nonumber
\end{align}
Plugging  the above estimate into \eqref{rankrEmP^l} with $l=1$, we obtain 
\begin{align}
\mathbb{E}\mP=&-m_1\Big(\mathbb{E}h_{1,1}(1,0)+\sn zm_2\mathbb{E}\mb{y}_1^*\mG_1\bv_i+\mathbb{E}h_{2,2}(1,0)+\sn zm_2\mathbb{E}\mb{y}_2^*\mG_1^2\bv_i \nonumber\\
&-\sn \mathbb{E}\mb{y}_2^*\mG_1\bv_i+\sn \frac{m_1'}{m_1}\mb{y}_2^*\bv_i\Big)+O_\prec(N^{-\frac12}). \label{EmP}
\end{align}
Applying \eqref{rankresthts(1,0)} with $l=1$, we have,
\begin{align*}
&h_{1,1}(1,0)=-\sn zm_2\mb{y}_1^*\mG_1\bv_i+\onh,\\
&h_{2,2}(1,0)=-\sn zm_2\mb{y}_2^*\mG_1^2\bv_i-\sn (zm_2)'\mb{y}_2^*\mG_1\bv_i+\onh.
\end{align*}
We substitute the above two estimates into \eqref{EmP} and get
\begin{align}
\mathbb{E}\mP=\frac{m_1'}{m_1}\sn\mathbb{E}\mb{y}_2^*(\mG_1-m_1)\bv_i+\onh, \label{19071975}
\end{align}
where we used the last equation in (\ref{identitym1m2}). Applying the similar arguments as   \eqref{rankrEmP^l} and \eqref{rankrcumuexpansion^t} to the RHS of (\ref{19071975}), for $l=1$, we will get 
\begin{align*}
\mathbb{E}\mP&=-m_1'\sn\mathbb{E}\mb{y}_2^*(H\mG_1+zm_2\mG_1)\bv_i+\onh\\
&=-m_1'\big(\mathbb{E}h_{2,1}(1,0)+\sn zm_2\mathbb{E}\mb{y}_2^*\mG_1\bv_i\big)+\onh=\onh,
\end{align*}
where the last step follows from (\ref{rankresthts(1,0)}) with $(t,s)=(2,1)$. This proves (\ref{est_EmP}).

Next we turn to prove \eqref{est_EmPk}. By (\ref{rankrcumuexpansion^t}) and  Lemma \ref{rankrlemh_ts}, 
 we observe that 
\begin{align}
\sn \mathbb{E}\mb{y}_t^*H\mG_1^{t}\bv_i\mP^{l-1}
=\mathbb{E} h_{t,t}(1,0) + \mathbb{E} h_{t,t}(0,1) + \mathbb{E} h_{t,t}(1,2) +O_\prec(N^{-\frac12}), \qquad t=1,2. \label{19072402}
\end{align}
Further, using \eqref{rankresthts(1,0)} to the first term in the RHS of (\ref{19072402}), we see
\begin{align}
\sn \mathbb{E}\mb{y}_1^*H\mG_1\bv_i\mP^{l-1}
&= -zm_2 \sqrt{N} \mathbb{E}\mb{y}_1^*\mG_1 \bv_i \mathcal P^{l-1} + \mathbb{E} h_{1,1}(0,1) + \mathbb{E} h_{1,1}(1,2) +O_\prec(N^{-\frac12}) \label{19071981}
\end{align}
and
\begin{align}
\sn \mathbb{E}\mb{y}_2^*H\mG_1^2 \bv_i\mP^{l-1}
=& -zm_2 \sqrt{N} \mathbb{E}\mb{y}_2^*\mG_1^2 \bv_i \mathcal P^{l-1} -\sqrt{N} (z m_2)' \mathbb{E}\mb{y}_2^*\mG_1 \bv_i \mathcal P^{l-1} \nonumber\\
&+ \mathbb{E} h_{2,2}(0,1) + \mathbb{E} h_{2,2}(1,2) +O_\prec(N^{-\frac12}). \label{19071982}
\end{align}
Plugging (\ref{19071980}), (\ref{19071981}) and (\ref{19071982})
into \eqref{rankrEmP^l}, one gets
\begin{align}
\mathbb{E}(\mP^l)=&-m_1\Big(\mathbb{E}h_{1,1}(0,1)+\mathbb{E}h_{1,1}(1,2)+\mathbb{E}h_{2,2}(0,1)+\mathbb{E}h_{2,2}(1,2)\Big) \nonumber\\
&-\sqrt N \left((zm_2)'+1 \right) \mathbb{E} \mb{y}_2^* \mG_1 \bv_i \mathcal P^{l-1} + \frac{m_1'}{m_1} \sqrt N \mb{y}_2^* \bv_i \mathbb{E} \mathcal P^{l-1}  +O_\prec(N^{-\frac{1}{2}}). \label{19071983}
\end{align}
Using the last equation of (\ref{identitym1m2}), we further obtain from (\ref{19071983}) that
\begin{align}
\mathbb{E}(\mP^l)=&-m_1\Big(\mathbb{E}h_{1,1}(0,1)+\mathbb{E}h_{1,1}(1,2)+\mathbb{E}h_{2,2}(0,1)+\mathbb{E}h_{2,2}(1,2)\Big)\nonumber\\
&+\frac{m_1'}{m_1}\mathbb{E}\mb{y}_2^*(\mG_1-m_1)\bv_i\mP^{l-1}+O_\prec(N^{-\frac{1}{2}}). \label{19071989}
\end{align}
Similarly to (\ref{rankrEmP^l}), by using the first identity in (\ref{identitym1m2}), one can write 
\begin{align*}
\mG_1-m_1=-m_1H\mG_1-zm_1 m_2 \mG_1.
\end{align*}
Then,  using Lemma \ref{rankrlemh_ts}, we  get
\begin{align}
\frac{m_1'}{m_1}\mathbb{E}\mb{y}_2^*(\mG_1-m_1)\bv_i\mP^{l-1} &= -m_1' \mathbb{E} \sn \mb{y}_2^*\big( H\mG_1+z m_2 \mG_1\big)\bv_i\mP^{l-1} \nonumber\\
&=-m_1'\Big(\mathbb{E}h_{2,1}(0,1)+\mathbb{E}h_{2,1}(1,2)\Big)+O_\prec(N^{-\frac{1}{2}}). \label{19071988}
\end{align}
Plugging (\ref{19071988}) into (\ref{19071989}) yields 
\begin{align*}
\mathbb{E}(\mP^l)=&-m_1\Big(\mathbb{E}h_{1,1}(0,1)+\mathbb{E}h_{1,1}(1,2)+\mathbb{E}h_{2,2}(0,1)+\mathbb{E}h_{2,2}(1,2)\Big)\nonumber\\
&-m_1'\Big(\mathbb{E}h_{2,1}(0,1)+\mathbb{E}h_{2,1}(1,2)\Big)+O_\prec(N^{-\frac{1}{2}}).
\end{align*}
It remains to compute the explicit formula for the RHS of the above equation. First, using \rf{rankresthts(0,1)}, we get
\begin{align*}
&-m_1\Big(\mathbb{E}h_{1,1}(0,1)+\mathbb{E} h_{2,2}(0,1)\Big)-m_1' \mathbb{E}h_{2,1}(0,1)\nonumber\\
&=(l-1)\Big((\mb{y}_1^*\mb{y}_1+(\mb{y}_1^*\bv_i)^2)m_1^2(zm_1)'+(\mb{y}_2^*\mb{y}_1+\mb{y}_1^*\bv_i\mb{y}_2^*\bv_i) \big(m_1^2(zm_1)'\big)' \nonumber\\
&\quad+(\mb{y}_2^*\mb{y}_2+(\mb{y}_2^*\bv_i)^2)\big(m_1m_1'(zm_1)''+(m_1')^2(zm_1)'+\frac{1}{6}m_1^2(zm_1)'''\big) \Big) \mathbb{E} \mathcal P^{l-2}.
\end{align*}
Recall the definitions for $\mb{y}_{1}$ and $\mb{y}_2$ in \rf{defy12} and the matrix $\mathcal M_i$ in \eqref{def:Mdiag} and \eqref{def:Moff}.  By elementary calculation, we arrive at
\begin{align}\label{ViI}
&-m_1\Big(\mathbb{E}h_{1,1}(0,1)+\mathbb{E} h_{2,2}(0,1)\Big)-m_1' \mathbb{E}h_{2,1}(0,1)\nonumber\\
&=(l-1) \Big(\sum_{j=1}^{r+2}\mathcal{M}_i(j,j)c_j^2+2\mathcal{M}_i(i,r+2)c_ic_{r+2}\Big)\mathbb{E} \mathcal P^{l-2} \nonumber\\
&= (l-1) \mathbf{c}^* \mathcal M_i \mathbf{c} \mathbb{E} \mathcal P^{l-2}.
\end{align}
Next, by \eqref{rankresthts(1,2)}, we have
\begin{align*}
&-m_1\big(\mathbb{E}h_{1,1}(1,2)+\mathbb{E}h_{2,2}(1,2)\big)-m_1' \mathbb{E}h_{2,1}(1,2)\nonumber\\
&=(l-1) \kappa_4 \Big( \mathbf{s}_{2,2}(\mb{y}_1,\bv_i)(zm_2m_1^2)^2 + \mathbf{s}_{2,2}(\mb{y}_2,\bv_i)\big((zm_2m_1^2)'\big)^2 \\
&\qquad + 2\mathbf{s}_{1,1,2}(\mb{y}_1,\mb{y}_2,\bv_i)(zm_2m_1^2)(zm_2m_1^2)'  \Big) \mathbb{E} \mathcal P^{l-2},
\end{align*}
which, by the definitions of $\mb{y}_{1}$ and $\mb{y}_2$ in \rf{defy12} and the matrix $\mathcal K$ in \eqref{def:K}, can be simplified to
\begin{align}\label{ViJ}
&-m_1\Big(\mathbb{E}h_{1,1}(1,2)+\mathbb{E}h_{2,2}(1,2)\Big)-m_1' \mathbb{E}h_{2,1}(1,2)\nonumber\\
&=(l-1)\kappa_4\Big(\sum_{j=1}^{r+2}\mathcal{K}_i(j,j)c_j^2+\sum_{1\le j\neq k\le r+2} 2\mathcal{K}_i (j,k) c_{j}c_{k} \Big) \mathbb{E} \mathcal P^{l-2}\nonumber\\
&=(l-1) \kappa_4 \mathbf{c}^* \mathcal K_i \mathbf{c} \mathbb{E} \mathcal P^{l-2}.
\end{align}
Combining \rf{ViI} and $\rf{ViJ}$, we complete the proof of  \eqref{est_EmPk} in Proposition \ref{recursivemP}. Hence, we conclude the proof of  Proposition \ref{recursivemP}. 
\end{proof}
The rest of this section is devoted to the proof of Lemma \ref{rankrlemh_ts}. It is convenient to first introduce the next lemma, which will be used  to control the negligible terms in the proof of Lemma \ref{rankrlemh_ts}.
\begin{lem}\label{Importantests}
For a fixed integer $n\geq 1$, let $\mb{\eta}_i=(\eta_{i1},\ldots, \eta_{iM})^T\in \mathbb{C}^M, i\in \lb 0, n\rb$ be any given deterministic vectors with $\max_i\|\mathbf{\eta}_i\|\leq C$ for some positive constant $C$.  For any positive integers $s_0,s_1,\cdots,s_n$ and $a=0,1$, we have the following estimates:
\begin{align}
&\sum_{q}\Big|\eta_{0q}\big((\mG_1^{s_0})_{qq}\big)^a\Big(\prod_{t=1}^n(\mG_1^{s_t}\mb{\eta}_t)_q\Big)\Big|=O_\prec(1)\label{Importestsum1},\\
&\Big|\sum_{k}\big((X^*\mG_1^{s_0}X)_{kk}\big)^a(X^*\mG_1^{s_1}\mb{\eta}_1)_{k}\Big|=O_\prec(1). \label{Importestsum2}
\end{align}
\end{lem}
\begin{proof}[Proof of Lemma \ref{Importantests}]
The first estimate \rf{Importestsum1} can be proved by \eqref{eq:basicG} and the isotropic local law \eqref{est_DG} as follows: 
\begin{align*}
&\sum_{q}\Big|\eta_{0q}\big((\mG_1^{s_0})_{qq}\big)^a\Big(\prod_{t=1}^n(\mG_1^{s_t}\mb{\eta}_t)_q\Big)\Big|=\sum_{q}\Big|\eta_{0q}\big((\mG_1^{s_0})_{qq}\big)^a\Big(\prod_{t=1}^n\langle \mathbf{e}_q, \frac{\mG_1^{(s_t-1)}}{(s_t-1)!}\mb{\eta}_t \rangle\Big)\Big|\\
&=\sum_q\Big|\eta_{0q}\Big(\prod_{t=1}^n\eta_{tq}\Big)\Big(\Big(\frac{m_1^{(s_0-1)}}{(s_0-1)!}+O_\prec\big(\frac{1}{\sn}\big)\Big)^a\prod_{t=1}^n\Big(\frac{m_1^{(s_t-1)}}{(s_t-1)!}+O_\prec\big(\frac{1}{\sn}\big)\Big)\Big)\Big|=O_\prec(1).
\end{align*}
To prove \rf{Importestsum2}, we notice that by \rf{est_XG},
\begin{align*}
\Big|\sum_{k}(X^*\mG_1^{s_1}\mb{\eta}_1)_{k}\Big|=\sn \mb{1}_N^*X^*\mG_1^{s_1}\mb{\eta}_1=O_\prec(1),
\end{align*}
where $\sqrt{N}\mb{1}_N\in \mathbb{R}^M$ is the all-ones vector. Further,  by \rf{relationXG} and the isotropic local law Theorem \ref{isotropic}, we get
\begin{align*}
&\Big|\sum_{k}\big((X^*\mG_1^{s_0}X)_{kk}\big)^a(X^*\mG_1^{s_1}\mb{\eta}_1)_{k}\Big|=\Big|\sum_{k}\big((\mG_2^{s_0-1}+z \mG_2^{s_0} )_{kk}\big)^a(X^*\mG_1^{s_1}\mb{\eta}_1)_{k}\Big|\\
&=\Big|\sum_{k}\Big(\frac{m_2^{(s_0-2)}}{(s_0-2)!}+\frac{zm_2^{(s_0-1)}}{(s_0-1)!}+O_\prec\big(N^{-\frac{1}{2}}\big)\Big)^a (X^*\mG_1^{s_1}\mb{\eta}_1)_{k}\Big|=O_\prec(1),\quad a=0,1
\end{align*}
with the convention $m_2^{(-1)}/(-1)!=1$.
\end{proof}

With Lemma \ref{Importantests}, we can now prove Lemma \ref{rankrlemh_ts}.
\begin{proof} [Proof of Lemma \ref{rankrlemh_ts}] In this proof, we fix $l$ and $i$. 
First, we compute $h_{t,s}(1,0)$ for $t,s=1,2$. Recall the definition in \rf{rankrdefhts}. By \rf{deriXG^l}, we have
\begin{align*}
h_{t,s}(1,0)&=N^{-\frac{1}{2}}\sum_{q,k}y_{tq}\frac{\partial (X^*\mG_1^s\bv_i)_{k}}{\partial x_{qk}}\mP^{l-1}\\
&= N^{-\frac{1}{2}}\sum_{q,k} y_{tq} \Big( (\mG_1^s\bv_i)_{q}-\sum_{a=1}^2\sum_{\begin{subarray}{c}s_1,s_2\geq 1\\ s_1+s_2=s+1 \end{subarray}}(X^*\mG_1^{s_1}\mathscr{P}_a^{qk}\mG_1^{s_2}\bv_i)_{k} \Big) \mP^{l-1},
\end{align*}
where $\mathscr{P}_a^{qk}, a=1,2$ are defined in \rf{P012}. Taking the sum over $q,k$, we get 
\begin{align}
h_{t,s}(1,0)
&=\sqrt N \mb{y}_t^*\mG_1^s\bv_i - N^{-\frac12} \sum_{\begin{subarray}{c}s_1,s_2\geq 1\\s_1+s_2=s+1\end{subarray}} \Big(\mb{y}_t^* \mG_1^{s_1} X X^* \mG_1^{s_2} \bv_i + \mb{y}_t^* \mG_1^{s_2} \bv_i \; \text{Tr}(X^* \mG_1^{s_1} X)\Big)\mP^{l-1}\label{19072011}.
\end{align}
By the identity in \eqref{relationXG},  Theorem \ref{isotropic}, and the fact $|\mP|\prec 1$, we further get 
\begin{align}
&N^{-\frac12} \sum_{\begin{subarray}{c}s_1,s_2\geq 1\\s_1+s_2=s+1\end{subarray}} \mb{y}_t^* \mG_1^{s_1} X X^* \mG_1^{s_2} \bv_i \mP^{l-1} \nonumber \\
& = N^{-\frac12} \sum_{\begin{subarray}{c}s_1,s_2\geq 1\\s_1+s_2=s+1\end{subarray}} \mb{y}_t^* (\mG_1^{s_1+s_2-1} + z \mG_1^{s_1+s_2}) \bv_i \mP^{l-1}=O_\prec(N^{-\frac12}). \label{19072010}
\end{align}
Substituting (\ref{19072010}) into (\ref{19072011}), we get 
\begin{align}
h_{t,s}(1,0)&=\sqrt N \mb{y}_t^*\mG_1^s\bv_i - \sqrt N \sum_{\begin{subarray}{c}s_1,s_2\geq 1\\s_1+s_2=s+1\end{subarray}} \mb{y}_t^* \mG_1^{s_2} \bv_i \frac{1}{N} \text{Tr}(X^* \mG_1^{s_1} X)\mP^{l-1} + O_\prec(N^{-\frac12}).  \label{19072015}
\end{align}
Using the second identity in (\ref{relationXG}), and Theorem \ref{isotropic}, we have 
\begin{align}
\frac{1}{N}\text{Tr}(X^* \mG_1^{s_1} X)= \frac{1}{N}\text{Tr}(\mG_2^{s_1-1} + z \mG_2^{s_1})=\frac{m_2^{(s_1-2)}}{(s_1-2)!}+\frac{zm_2^{(s_1-1)}}{(s_1-1)!}+O_\prec(N^{-1}). 
\end{align}
Then we  obtain from (\ref{19072015}) that
\begin{align}
h_{t,s}(1,0)=&-\sn z m_2 \mb{y}_t^*\mG_1^s\bv_i \mP^{l-1}  \nonumber\\
&-\sn \sum_{\begin{subarray}{c}2\le s_1\le s, 1\le s_2\le s-1\\s_1+s_2=s+1\end{subarray}}\Big(\frac{m_2^{(s_1-2)}}{(s_1-2)!}+\frac{zm_2^{(s_1-1)}}{(s_1-1)!}\Big)\mb{y}_t^*\mG_1^{s_2}\bv_i\mP^{l-1}+O_\prec(N^{-\frac12}). \label{19072020}
\end{align}
Further by the simple fact
\begin{align}\label{eq:simple}
\frac{z m_a^{(l)}}{l!} + \frac{m_a^{(l-1)}}{(l-1)!} = \frac{1}{l!} (z m_a)^{(l)}, \qquad a=1,2,
\end{align}
we can conclude the proof of \eqref{rankresthts(1,0)} from (\ref{19072020}).

Next, we turn to estimate $h_{t,s}(0,1)$, which by the definition in \rf{rankrdefhts} reads
\begin{align}
h_{t,s}(0,1)&=(l-1)N^{-\frac{1}{2}}\sum_{q,k}y_{tq}(X^*\mG_1^s\bv_i)_{k}\frac{\partial \mP}{\partial x_{qk}}\mP^{l-2}.
\label{rankrhts(1,0)}
\end{align}
Using the formula \rf{eq:Pd1} to \rf{rankrhts(1,0)}, we get
\begin{align} 
h_{t,s}(0,1)
&=-(l-1)\sum_{b=1}^2\sum_{\begin{subarray}{c}b_1,b_2\geq 1\\b_1+b_2=b+1\end{subarray}}\Big((\mb{y}_b^*\mG_1^{b_1}\mb{y}_t)(\bv_i^*\mG_1^{b_2}XX^*\mG_1^{s}\bv_i)\nonumber\\
&\qquad\qquad\qquad\qquad\qquad+(\mb{y}_t^*\mG_1^{b_2}\bv_i)(\mb{y}_b^*\mG_1^{b_1}XX^*\mG_1^{s}\bv_i)\Big)\mP^{l-2}. \label{17092027}
\end{align}
By \eqref{relationXG}, we have
\begin{align*}
\mG_1^{b_a}XX^*\mG_1^{s} = \mG_1^{b_a+s-1} + z \mG_1^{b_a+s}, \qquad a=1,2.
\end{align*}
Hence, by  \eqref{est_DG} and the bound $|\mP|\prec 1$, one gets from (\ref{17092027}) that
\begin{align}
h_{t,s}(0,1)&=-(l-1)\sum_{b=1}^2\sum_{\begin{subarray}{c}b_1,b_2\geq 1\\b_1+b_2=b+1\end{subarray}}\bigg( \frac{m_1^{(b_1-1)}}{(b_1-1)!} \Big(\frac{m_1^{(s+b_2-2)}}{(s+b_2-2)!}+\frac{zm_1^{(s+b_2-1)}}{(s+b_2-1)!}\Big) \mb{y}_b^* \mb{y}_t \nonumber\\
&\quad + \frac{m_1^{(b_2-1)}}{(b_2-1)!} \Big(\frac{m_1^{(s+b_1-2)}}{(s+b_1-2)!}+\frac{zm_1^{(s+b_1-1)}}{(s+b_1-1)!}\Big)\mb{y}_t^*\bv_i\mb{y}_b^*\bv_i \bigg)\mP^{l-2}+O_\prec(N^{-\frac12}). \label{19072029}
\end{align}
Using \eqref{eq:simple} to (\ref{19072029}), we can conclude the proof of \rf{rankresthts(0,1)}.

Next, we show \rf{rankresthts(1,2)}. First,  by the definition in (\ref{rankrdefhts}), we have 
 \begin{align} 
 h_{t,s}(1,2)
 &=\frac{\kappa_4}{2N^{\frac{3}{2}}}\sum_{q,k}y_{tq}\frac{\partial (X^*\mG_1^s\bv_i)_{k}}{\partial x_{qk}}\Big((l-1) \frac{\partial^2 \mP}{\partial x_{qk}^2}\mP^{l-2} + (l-1)(l-2)\Big(\frac{\partial \mP}{\partial x_{qk}}\Big)^2\mP^{l-3} \Big)\nonumber\\
 &=: \kappa_4(l-1)\mathcal J_1 +\kappa_4(l-1)(l-2) \mathcal J_2, 
  \label{17092053}
 \end{align} 
 where $\kappa_4(l-1)\mathcal{J}_1$ and $\kappa_4(l-1)(l-2)\mathcal{J}_2$ correspond to the sum involving the first and the second terms in the parenthesis in the first step, respectively. 
 In the following, we shall give the  estimates of $\mathcal J_1$ and  $\mathcal J_2$.
  
To estimate $\mathcal J_1$, using
 (\ref{eq:Pd2}) and \eqref{deriXG^l},  
we see
  \begin{align}
 \mathcal J_1&=
 \frac{ 1}{2N^{\frac{3}{2}}}\sum_{q,k}y_{tq}\frac{\partial (X^*\mG_1^s\bv_i)_{k}}{\partial x_{qk}} \frac{\partial^2 \mP}{\partial x_{qk}^2}\mP^{l-2} \nonumber\\
 &=\frac{1}{N}\sum_{q,k}y_{tq} \Big((\mG_1^s\bv_i)_{q}-\sum_{a=1}^2\sum_{\begin{subarray}{c}s_1,s_2\geq 1\\ s_1+s_2=s+1 \end{subarray}}(X^*\mG_1^{s_1}\mathscr{P}_a^{qk}\mG_1^{s_2}\bv_i)_{k} \Big) \nonumber\\
 &\quad \times \sum_{b=1}^2\mb{y}_b^*\Big(\sum_{a_1,a_2=1}^2\sum_{\begin{subarray}{c}b_1,b_2,b_3\geq 1\\b_1+b_2+b_3=b+2\end{subarray}}\mG_1^{b_1}\mathscr{P}_{a_1}^{qk}\mG_1^{b_2}\mathscr{P}_{a_2}^{qk}\mG_1^{b_3}-\sum_{\begin{subarray}{c}b_1,b_2\geq 1\\b_1+b_2=b+1\end{subarray}}\mG_1^{b_1}\mathscr{P}_0^{qk}\mG_1^{b_2}\Big)\bv_i\mP^{l-2}. \label{17092030}
 \end{align} 
Further, we claim that
\begin{align}
 \mathcal J_1&=\frac{1}{N}\sum_{q,k}y_{tq}\Big((\mG_1^s\bv_i)_{q}-\sum_{\begin{subarray}{c}s_1,s_2\geq 1\\ s_1+s_2=s+1 \end{subarray}}(X^*\mG_1^{s_1}\mathscr{P}_2^{qk}\mG_1^{s_2}\bv_i)_{k} \Big) \nonumber\\
 &\quad \times \sum_{b=1}^2\mb{y}_b^*\Big(\sum_{\begin{subarray}{c}b_1,b_2,b_3\geq 1\\b_1+b_2+b_3=b+2\end{subarray}}\mG_1^{b_1}\mathscr{P}_{1}^{qk}\mG_1^{b_2}\mathscr{P}_{2}^{qk}\mG_1^{b_3}-\sum_{\begin{subarray}{c}b_1,b_2\geq 1\\b_1+b_2=b+1\end{subarray}}\mG_1^{b_1}\mathscr{P}_0^{qk}\mG_1^{b_2}\Big)\bv_i\mP^{l-2} + O_\prec(N^{-\frac12}). \label{17092031}
\end{align}
To see the reduction from (\ref{17092030}) to (\ref{17092031}), we notice that the terms absorbed in $O_\prec(N^{-\frac12})$ 
 always contain some quadratic forms of $X^*\mG_1^{a}$ as a factor, for some $a\ge1$. Then the isotropic local law \eqref{est_XG} can be applied to show that these terms are bounded by $O_\prec(N^{-\frac12})$. For instance,  we have 
\begin{align*}
&N^{-\frac32}\sum_{q,k}y_{tq}  \Big(\sum_{\begin{subarray}{c}s_1,s_2\geq 1\\ s_1+s_2=s+1 \end{subarray}}(X^*\mG_1^{s_1}\mathscr{P}_1\mG_1^{s_2}\bv_i)_{k} \Big)  \frac{\partial^2 \mP}{\partial x_{qk}^2}\mP^{l-2}\\
&\prec N^{-\frac32}\bigg| \sum_{q,k}\sum_{\begin{subarray}{c}s_1,s_2\geq 1\\ s_1+s_2=s+1 \end{subarray}} y_{tq}  (X^*\mG_1^{s_1})_{kq} (X^*\mG_1^{s_2}\bv_i)_{k}\frac{\partial^2 \mP}{\partial x_{qk}^2}\bigg| \\
&\prec N^{-1} \sum_{\begin{subarray}{c}s_1,s_2\geq 1\\ s_1+s_2=s+1 \end{subarray}} \Big(\sum_{q} \big| y_{tq} \big| \Big)\sum_{k} \big|(X^*\mG_1^{s_1})_{kq} \big| \big|(X^*\mG_1^{s_2} \bv_i)_k\big|=O_\prec(N^{-\frac12}).
\end{align*}
In the first and second steps above, we used the  bound $|\mathcal P| \prec 1$ and (\ref{19072411}),  
 respectively. In the last step, we used the isotropic local law \eqref{est_XG} and also the fact  $\sum_{q}\big|y_{tq}\big|\prec \sn$. 
 The other negligible terms can be estimated similarly. We omit the details. 

Plugging the definitions in (\ref{P012}) into (\ref{17092031}) yields
 \begin{align*}
\mathcal J_1 &= \frac{1}{N}\sum_{q,k}y_{tq}\Big((\mG_1^s\bv_i)_{q}-\sum_{\begin{subarray}{c}s_1,s_2\geq 1\\s_1+s_2=s+1\end{subarray}}(X^*\mG_1^{s_1}X)_{kk}(\mG_1^{s_2}\bv_i)_{q}\Big)\nonumber\\
 & \times \sum_{b=1}^2\bigg(\sum_{\begin{subarray}{c}b_1,b_2,b_3\geq 1\\b_1+b_2+b_3=b+2\end{subarray}}(\mb{y}_b^*\mG_1^{b_1})_q(X^*\mG_1^{b_2}X)_{kk}(\mG_1^{b_3}\bv_i)_q-\sum_{\begin{subarray}{c}b_1,b_2\geq 1\\b_1+b_2=b+1\end{subarray}}(\mb{y}_b^*\mG_1^{b_1})_q(\mG_1^{b_2}\bv_i)_q\bigg)\mP^{l-2}+O_\prec(N^{-\frac12}).
 \end{align*}
Again, applying the identities in Lemma \ref{lemrelation}, the isotropic local laws \eqref{est_DG}, and \eqref{est_m12N}, we see 
 \begin{align*}
\mathcal J_1&=\bigg(\frac{m_1^{(s-1)}}{(s-1)!}-\sum_{\begin{subarray}{c}s_1,s_2\geq 1\\s_1+s_2=s+1\end{subarray}}\Big(\frac{m_2^{(s_1-2)}}{(s_1-2)!}+z\frac{m_2^{(s_1-1)}}{(s_1-1)!}\Big)\frac{m_1^{(s_2-1)}}{(s_2-1)!}\bigg)\sum_{b=1}^2\sum_{q}y_{tq}y_{bq}(v_i^q)^2\nonumber\\
 &\quad \times \bigg(\sum_{\begin{subarray}{c}b_1,b_2,b_3\geq 1\\b_1+b_2+b_3=b+2\end{subarray}}
\frac{m_1^{(b_1-1)}}{(b_1-1)!}\frac{m_1^{(b_3-1)}}{(b_3-1)!}\Big(\frac{m_2^{(b_2-2)}}{(b_2-2)!}+z\frac{m_2^{(b_2-1)}}{(b_2-1)!}\Big)-\frac{(m_1^2)^{(b-1)} }{(b-1)!}\bigg)\mP^{l-2},
 \end{align*}
 which by \eqref{eq:simple} 
and  the product rule can be rewritten as 
\begin{align}
\mathcal J_1&= - \frac{(zm_2m_1)^{(s-1)}}{(s-1)!}
\bigg( \Big(\sum_{q}y_{tq}y_{1q}v_{iq}^2\Big)(zm_2m_1^2)+ \Big(\sum_{q}y_{tq}y_{2q}v_{iq}^2\Big)(zm_2m_1^2)'\bigg)\mP^{l-2} +O_\prec(N^{-\frac12})\nonumber\\
&=-\frac{(zm_2m_1)^{(s-1)}}{(s-1)!}
\sum_{b=0}^1 \mathbf{s}_{1,1,2} (\mb{y}_t,\mb{y}_{b+1},\bv_i)(zm_2m_1^2)^{(b)} \mP^{l-2}+O_\prec(N^{-\frac12}). \label{17092054}
 \end{align}
 Next, we show that
 \begin{align}
 \mathcal J_2 =\frac{1}{2N^{\frac{3}{2}}}\sum_{q,k}y_{tq}\frac{\partial (X^*\mG_1^s\mb{v}_i)_{k}}{\partial x_{qk}}\Big(\frac{\partial \mP}{\partial x_{qk}}\Big)^2\mP^{l-3} = O_\prec(N^{-\frac12}). \label{17092051}
 \end{align}

With \eqref{deriXG^l}, we write $ \mathcal J_2$ as 
\begin{align}
\mathcal J_2&=\frac{1}{2N^{\frac{3}{2}}}\sum_{q,k}y_{tq}\bigg((\mG_1^s\bv_i)_{q}-\sum_{a=1}^2\sum_{\begin{subarray}{c} s_1,s_2\geq 1\\s_1+s_2=s+1\end{subarray}}(X^*\mG_1^{s_1}\mathscr{P}_a^{qk}\mG_1^{s_2}\bv_i)_{k}\bigg)\Big(\frac{\partial \mP}{\partial x_{qk}}\Big)^2\mP^{l-3} \nonumber\\
&=\frac{1}{2N^{\frac{3}{2}}}\sum_{q,k}y_{tq}\bigg((\mG_1^s \bv_i)_{q}-\sum_{\begin{subarray}{c} s_1,s_2\geq 1\\s_1+s_2=s+1\end{subarray}}(X^*\mG_1^{s_1}X)_{kk}(\mG_1^{s_2}\bv_i)_{q}\bigg)\Big(\frac{\partial \mP}{\partial x_{qk}}\Big)^2\mP^{l-3}+\onh, \label{17092052}
\end{align}
where in the last step we bounded the $a=1$ terms by  $O_\prec(N^{-\frac12})$ since 
\begin{align*}
&\bigg|\frac{1}{2N^{\frac{3}{2}}}\sum_{q,k}y_{tq}\sum_{\begin{subarray}{c} s_1,s_2\geq 1\\s_1+s_2=s+1\end{subarray}}(X^*\mG_1^{s_1})_{kq}(X^*\mG_1^{s_2}\bv_i)_{kq}\Big(\frac{\partial \mP}{\partial x_{qk}}\Big)^2\mP^{l-3}\bigg|\\
&\prec N^{-\frac32}\sum_{\begin{subarray}{c} s_1,s_2\geq 1\\s_1+s_2=s+1\end{subarray}}\sum_{q,k}\big|(X^*\mG_1^{s_1})_{kq}\big|\big|(X^*\mG_1^{s_2}\bv_i)_{kq}\big|=\onh.
\end{align*}
Here we used the  $O_\prec(1)$ bound for both $\mP$ and $\partial\mP/\partial x_{qk}$ (\cf (\ref{19072410})) for the first step and  \eqref{est_XG} for the second step.

Further, we have
\begin{align}
&\Big|N^{-\frac32}\sum_{q,k}y_{tq}\big((X^*\mG_1^{s_1}X)_{kk}\big)^a(\mG_1^{s_2}\bv_i)_q\Big(\frac{\partial \mP}{\partial x_{qk}}\Big)^2\mP^{l-3}\Big|\nonumber\\
&\prec N^{-\frac12} \sum_{q} \big|y_{tq}(\mG_1^{s_2}\bv_i)_q\big|=\onh,\qquad a=0,1. \label{19072050}
\end{align}
Here, in the first step we used  $O_\prec(1)$ bounds for $\mP$, $\partial\mP/\partial x_{qk}$ (\cf (\ref{19072410})),  and also $(X^*\mG_1^{s_1}X)_{kk}^a$ whose bound follows from (\ref{relationXG}) and the local law, and in the second step we use Cauchy-Schwarz inequality and the local law. With (\ref{19072050}), we can now conclude (\ref{17092051}) from (\ref{17092052}).  Then (\ref{17092053}), (\ref{17092054}) and (\ref{17092051}) imply \rf{rankresthts(1,2)}.

In the sequel,  we prove (2) of Lemma \ref{rankrlemh_ts}, i.e.,   we show that except for (\ref{rankresthts(1,0)})-(\ref{rankresthts(1,2)}) all the other terms $h_{t,s}(\alpha_1,\alpha_2)$ with $\alpha_1+\alpha_2\le 3$ can be bounded by $O_\prec(N^{-\frac12})$.  We start with the case when $\alpha_1+\alpha_2=2$, i.e. $(\alpha_1,\alpha_2)=(2,0),(0,2),(1,1)$.  Recall the notation $\kappa_3$ for the common $3$rd cumulant of all $\sqrt{N}x_{ij}$'s from Assumption \ref{assumption} (iii). 
First,  by (\ref{eq:2ndXG}), we have 
\begin{align}\label{eq:hst20}
h_{t,s}(2,0)&= \frac{\kappa_3 }{2N}\sum_{q,k}y_{tq}\frac{\partial^2 (X^*\mG_1^s\bv_i)_{k}}{\partial x_{qk}^2}\mP^{l-1}\nonumber\\
&=\frac{\kappa_3 }{N}\sum_{q,k}y_{tq}\bigg(\sum_{a_1,a_2=1}^2\sum_{\begin{subarray}{c}s_1,s_2,s_3\geq 1\\s_1+s_2+s_3=s+2\end{subarray}}\Big(X^*\mG_1^{s_1}\mathscr{P}_{a_1}^{qk}\mG_1^{s_2}\mathscr{P}_{a_2}^{qk}\mG_1^{s_3}\bv_i\Big)_{k}\nonumber\\
&\quad -\sum_{a=1}^2\sum_{\begin{subarray}{c}s_1,s_2\geq 1\\s_1+s_2=s+1\end{subarray}}\Big(\mG_1^{s_1}\mathscr{P}_{a}^{qk}\mG_1^{s_2}\bv_i\Big)_{q}-\sum_{\begin{subarray}{c}s_1,s_2\geq 1\\s_1+s_2=s+1\end{subarray}}\Big(X^*\mG_1^{s_1}\mathscr{P}_0^{qk}\mG_1^{s_2}\bv_i\Big)_{k}\bigg)\mP^{l-1}.
\end{align}
Note that all terms above contain at least one quadratic form of $X^*\mG_1^a$ as a factor,  for some $a\ge 1$. This fact eventually leads to the $O_\prec(N^{-\frac12})$ bound for all terms above, by the isotropic local law.  More specifically, plugging the definitions in (\ref{P012}) into (\ref{eq:hst20}) and taking the sums, we can see that  the RHS of (\ref{eq:hst20}) is a linear combination of the terms of the following forms 
\begin{align}
&N^{-1} \sum_{q,k}y_{tq}(X^*\mG_1^{s_1})_{kq}(X^*\mG_1^{s_2})_{kq}(X^*\mG_1^{s_3}\bv_i)_k \mP^{l-1} \label{17092071}\\
&N^{-1} \sum_{q,k}y_{tq}\big((X^*\mG_1^{s_1}X)_{kk}\big)^a(X^*\mG_1^{s_2})_{kq}(\mG_1^{s_3}\bv_i)_q \mP^{l-1},\label{17092072}\\ 
&N^{-1} \sum_{q,k}y_{tq}\big((X^*\mG_1^{s_1}X )_{kk}\big)^a(\mG_1^{s_2})_{qq}(X^*\mG_1^{s_3}\bv_i)_{k} \mP^{l-1}, \qquad a=1,2. \label{17092070}
\end{align}
First, by simply using the $O_\prec(1)$ bound for $\mP^{l-1}$ and the $O_\prec(N^{-\frac12})$ bound for  $(X^*\mG_1^{s_1})_{kq}$, $(X^*\mG_1^{s_2})_{kq}$, and $(X^*\mG_1^{s_3}\bv_i)_k$ one can get the  $\onh$ bound for the  term in (\ref{17092071}).  Second, using the $O_\prec(1)$ bound for $\mP^{l-1}$ and $\big((X^*\mG_1^{s_1}X)_{kk}\big)^a$, and also the $O_\prec(N^{-\frac12})$ bound for  $(X^*\mG_1^{s_2})_{kq}$, we have 
\begin{align}
&(\ref{17092072})=O_\prec\Big(N^{-\frac12}\sum_{q}\big|y_{tq}(\mG_1^{s_3}\bv_i)_q\big|\Big)=O_\prec(N^{-\frac12}),
\end{align}
where we used the Cauchy-Schwarz inequality and the isotropic local law in the last step. 
Third, by Cauchy-Schwarz inequality and \eqref{est_DG}, we have
\begin{align}\label{eq:middleCS}
&\Big|N^{-1}\sum_{q,k}y_{tq}\big((X^*\mG_1^{s_1}X )_{kk}\big)^a(\mG_1^{s_2})_{qq}(X^*\mG_1^{s_3}\bv_i)_{k}\Big|\nonumber\\
&\leq N^{-1}\|\mathbf{y}_t\|\Big(\sum_{q}(\mG_1^{s_2})_{qq}^2\Big)^{1/2}\Big|\sum_{k}\big((X^*\mG_1^{s_1}X )_{kk}\big)^a(X^*\mG_1^{s_3}\bv_i)_{k}\Big|\nonumber\\
&\prec N^{-\frac12}\Big|\sum_{k}\big((X^*\mG_1^{s_1}X )_{kk}\big)^a(X^*\mG_1^{s_3}\bv_i)_{k}\Big|\prec N^{-\frac12}.
\end{align}
In the last step above, we used \eqref{Importestsum2}. Hence, we conclude $h_{t,s}(2,0)=\onh$.

Next, in the case of  $(\alpha_1,\alpha_2)=(0,2)$, by the definition in (\ref{rankrdefhts}), we have 
\begin{align}
 h_{t,s}(0,2)
 &=(l-1)(l-2)\frac{\kappa_3 }{2N}\sum_{q,k}y_{tq}(X^*\mG_1^s\bv_i)_{k}\Big(\frac{\partial\mP}{\partial x_{qk}}\Big)^2 \mP^{l-3}\nonumber\\
 &\qquad+(l-1)\frac{\kappa_3 }{2N}\sum_{q,k}y_{tq}(X^*\mG_1^s\bv_i)_{k}\frac{\partial^2\mP}{\partial x_{qk}^2} \mP^{l-2}. \label{rankrhts(0,2)}
\end{align}
We then use the formula  in \rf{eq:Pd1}. After expanding the first term in the RHS of \rf{rankrhts(0,2)}, one notices that it can be written as a linear combination of the terms of the  forms 
\begin{align}
\sum_{q}y_{tq}(\mG_1^{a_1}\mb{\eta}_1)_{q}(\mG_1^{a_2}\mb{\psi}_1)_{q}\sum_{k}(X^*\mG_1^s\bv_i)_{k}(X^*\mG_1^{b_1}\mb{\eta}_2)_{k}(X^*\mG_1^{b_2}\mb{\psi}_2)_{k}\mP^{l-3}, \label{19072080}
\end{align}
where the vectors $\mb{\eta}_\alpha,{\mb{\psi}}_\alpha$ $(\alpha=1,2)$ take $\bv_i$, $\mb{y}_1$ or $\mb{y}_2$ and $a_1,a_2,b_1,b_2=1$ or 2.
We claim that  the general form in (\ref{19072080}) is $\onh$ for all choices of $\mb{\eta}_\alpha,{\mb{\psi}}_\alpha$, $a_1,a_2,b_2,b_2$ listed above. To see this, we first notice that  the isotropic local law  \eqref{est_XG} implies
\begin{align}
\sum_{k}(X^*\mG_1^s\bv_i)_{k}(X^*\mG_1^{b_1}\mb{\eta}_2)_{k}(X^*\mG_1^{b_2}\mb{\psi}_2)_{k}=\onh. \label{19072081}
\end{align}
 Further, from \rf{Importestsum1}, we have 
\begin{align}
\sum_{q}y_{tq}(\mG_1^{a_1}\mb{\eta}_1)_{q}(\mG_1^{a_2}\mb{\psi}_1)_{q}=O_\prec(1). \label{19072082}
\end{align} 
Combining (\ref{19072081}), (\ref{19072082}), and the fact $|\mP|\prec 1$, we conclude that the term in (\ref{19072080}) is of order $O_\prec(N^{-\frac12})$. This further implies that the first term in the RHS of  \rf{rankrhts(0,2)} is $O_\prec(N^{-\frac12})$.

 Analogously,  using the formula in  \rf{eq:Pd2}, it is easy to see that the second term in the RHS of \rf{rankrhts(0,2)} is a linear combination of the terms of the following forms
\begin{align}
&N^{-\frac12}\Big(\sum_{q}y_{tq}(\mG_1^{b_1}\mb{\eta}_1)_q(\mG_1^{b_2}\mb{\eta}_2)_q\Big)\Big(\sum_{k}\big((X^*\mG_1^{b_3}X)_{kk}\big)^a(X^*\mG_1^s\bv_i)_{k}\Big)\mP^{l-2},\label{eq:h02-1}\\
&N^{-\frac12} \Big(\sum_{q}y_{tq}(\mG_1^{b_1}\mb{\eta}_1)_q\Big)\Big(\sum_{k}(X^*\mG_1^{b_2})_{kq}(X^*\mG_1^{b_3}\mb{\eta}_2)_k(X^*\mG_1^s\bv_i)_{k}\Big)\mP^{l-2},\label{eq:h02-2}\\
&N^{-\frac12} \Big(\sum_{q}y_{tq}(\mG_1^{b_1})_{qq}\Big)\Big(\sum_{k}(X^*\mG_1^{b_2}\mb{\eta}_1)_{k}(X^*\mG_1^{b_3}\mb{\eta}_2)_k(X^*\mG_1^s\bv_i)_{k}\Big)\mP^{l-2}, \label{eq:h02-3}
\end{align}
for $\mb{\eta}_1, \mb{\eta}_2=\bv_i, \mb{y}_1$ or $\mb{y}_2$, $a=0,1$ and $b_1,b_2,b_3=1$ or 2. We claim that all of the above terms can be bounded by $O_\prec(N^{-\frac12})$. First, recall the bound $|\mathcal{P}|\prec 1$.  The $O_\prec(N^{-\frac12})$ bound for \eqref{eq:h02-1} follows directly from \rf{Importestsum1} and \rf{Importestsum2}. From \eqref{Importestsum1} and the isotropic local law \eqref{est_XG}, we see that \eqref{eq:h02-2} is also bounded by $O_\prec(N^{-\frac12})$. The estimate of \eqref{eq:h02-3} is similar to \eqref{eq:middleCS} by using the Cauchy-Schwarz inequality and the isotropic local law \eqref{est_XG}. We thus omit the details. Hence, the second term in the RHS of \rf{rankrhts(0,2)} is also of order $O_\prec(N^{-\frac12})$. 
This together with the same bound for the first term in the RHS of \rf{rankrhts(0,2)} leads to the fact that $$h_{t,s}(0,2)=O_\prec(N^{-\frac12}).$$

Next, we turn to  $h_{t,s}(1,1)$.  By the definition in (\ref{rankrdefhts}), we have 
\begin{align}
{h}_{t,s}(1,1)=\frac{\kappa_3}{N}\sum_{q,k}y_{tq}\frac{\partial (X^*\mG_1^s\bv_i)_{k}}{\partial x_{qk}}\frac{\partial \mP}{\partial x_{qk}}\mP^{l-2}.\label{rankrhts(1,1)}
\end{align}
Using \rf{deriXG^l} and \rf{eq:Pd1}, we can write
\begin{align}
h_{t,s}(1,1)&=-N^{-\frac12}{\kappa_3}\sum_{q,k}y_{tq}\Big((\mG_1^s\bv_i)_{q}-\sum_{\begin{subarray}{c}s_1,s_2\geq 1\\s_1+s_2=s+1\end{subarray}}\Big((X^*\mG_1^{s_1})_{kq}(X^*\mG_1^{s_2}\bv_i)_{k}+(X^*\mG_1^{s_1}X)_{kk}(\mG_1^{s_2}\bv_i)_{q}\Big)\Big)
\nonumber\\
&\qquad\qquad\times \sum_{b=1}^2\sum_{\begin{subarray}{c}b_1,b_2\geq 1\\b_1+b_2=b+1\end{subarray}}\Big((\mb{y}_b^*\mG_1^{b_1})_q(X^*\mG_1^{b_2}\bv_i)_k+(X^*\mG_1^{b_1}\mb{y}_b)_k(\mG_1^{b_2}\bv_i)_q\Big)\mP^{l-1}.
\end{align}
It is easy to see that the above is a linear combination of the terms of  the following forms
\begin{align*}
&N^{-\frac12} \Big(\sum_{q}y_{tq}(\mG_1^{s_1}\bv_i)_{q}(\mG_1^{b_1}\mb{\eta}_1)_{q}\Big)\Big(\sum_k\big((X^*\mG_1^{s_2}X)_{kk}\big)^a(X^*\mG_1^{b_2}\mb{\eta}_2)_{k}\Big)\mP^{l-1},\nonumber\\
&N^{-\frac12} \Big(\sum_{q}y_{tq}(\mG_1^{b_1}\mb{\eta}_1)_{q}\Big)\Big(\sum_k(X^*\mG_1^{s_1})_{kq}(X^*\mG_1^{s_2}\mb{v}_i)_k(X^*\mG_1^{b_2}\mb{\eta}_2)_{k}\Big)\mP^{l-1}, \qquad a=1,2.
\end{align*}
Similarly to estimates of \eqref{eq:h02-1} an \eqref{eq:h02-2}, both two terms above can be bounded by $O_\prec(N^{-\frac12})$. Hence, we have $h_{t,s}(1,1)=O_\prec(N^{-\frac12})$. 

Next, we  consider the other cases for $\alpha_1+\alpha_3=3$ except for (\ref{rankresthts(1,2)}), i.e. $(\alpha_1,\alpha_2)=(3,0),(2,1),(0,3)$.  We start with the formulas  in (\ref{rankrorder3derivative}) and  (\ref{eq:Pd3}). Notice that according to the definition in (\ref{def of O}), we have $\mathcal{O}_2=\{(0,1),(0,2),(1,0),(2,0)\}$ in (\ref{rankrorder3derivative}) and  (\ref{eq:Pd3}). 

With (\ref{rankrorder3derivative}), we  can now estimate   $h_{t,s}(3,0)$. By the definition in \rf{rankrdefhts},
\begin{align}
h_{t,s}(3,0)=\frac{\kappa_4 }{3!N^{\frac{3}{2}}}\sum_{q,k}y_{tq}\frac{\partial^3 (X^*\mG_1^s\bv_i)_{k}}{\partial x_{qk}^3}\mP^{l-1}. \label{eq:h(3,0)}
\end{align}
After plugging in \eqref{rankrorder3derivative} and taking the sums, one can check that except for the following type of terms  
\begin{align}
N^{-\frac32}\Big(\sum_{q}y_{tq}(\mG_1^{s_3})_{qq}(\mG_1^{s_4}\bv_i)_q\Big)\Big(\sum_k\big((X^*\mG_1^{s_1}X)_{kk}\big)^a\big((X^*\mG_1^{s_2}X)_{kk}\big)^b\Big)\mathcal P^{l-1} \qquad  a,b=0,1, \label{19072101}
\end{align} 
all the other terms of (\ref{eq:h(3,0)}) contain at least  one quadratic form of $X^*\mG_1^a$ for some $a\ge 1$. Actually, by the isotropic local law \eqref{est_XG}, those terms with at least  one quadratic form of $X^*\mG_1^a$ can all be bounded by $O_\prec(N^{-1})$. For instance, 
\begin{align*}
&N^{-\frac32} \sum_{q,k} y_{tq} \big(\mG_1^{s_1}\mathscr{P}_1^{qk}\mG_1^{s_2}\mathscr{P}_1^{qk}\mG_1^{s_3}\bv_i\big)_{q}\mP^{l-1} \\
&= N^{-\frac32} \sum_{q,k} y_{tq} (\mG_1^{s_1})_{qq} (X^*\mG_1^{s_2})_{kq} (X^*\mG_1^{s_3} \bv_i)_k \mP^{l-1} =O_\prec(N^{-1}).
\end{align*}
The other terms with at least  one quadratic form of $X^*\mG_1^a$ can be estimated similarly.  We omit the details. 
Further, using \rf{Importestsum1} and also Remark \ref{boundrmk}, one can easily get the $O_\prec(N^{-\frac12})$ for the terms in (\ref{19072101}). As a consequence, we get  $h_{t,s}(3,0)=\onh$.

Next, for $h_{t,s}(2,1)$, by the definition in (\ref{rankrdefhts}), we have
\begin{align}
h_{t,s}(2,1)=\frac{(l-1)\kappa_4}{2N^{\frac{3}{2}}}\sum_{q,k}y_{tq}\frac{\partial^2 (X^*\mG_1^s\bv_i)_{k}}{\partial x_{qk}^2}\frac{\partial \mP}{\partial x_{qk}}\mP^{l-2}.\label{rankrhts(2,1)}
\end{align}
Using the formula in (\ref{eq:2ndXG}) and further the $O_\prec(N^{-\frac12})$ bound for  the quadratic forms of $X^*\mG_1^a, a\geq 1$, it is not difficult to see 
 $$ \frac{\partial^2 (X^*\mG_1^s\bv_i)_{k}}{\partial x_{qk}^2}=\onh, $$
similarly to the previous discussion.   Then, the above bound together with the $O_\prec(1)$ bound for $\partial \mP/\partial x_{qk}$ (\cf (\ref{19072410})) and $\mathcal{P}$,  one can conclude that  $h_{t,s}(2,1)=\onh$.

For $h_{t,s}(0,3)$, we first write
\begin{align*}
h_{t,s}(0,3)&=\frac{\kappa_4 }{3!N^{\frac{3}{2}}}\sum_{q,k}y_{tq}(X^*\mG_1^s\bv_i)_{k}\frac{\partial^3 \mP^{l-1}}{\partial x_{qk}^3}\nonumber\\
&=\mathcal \kappa_4(l-1)(l-2)(l-3)\mathcal{L}_1 + \kappa_4(l-1)(l-2)\mathcal L_2 + \kappa_4(l-1)\mathcal L_3, 
\end{align*}
where
\begin{align*}
&\mathcal L_1:= \frac{1 }{3!N^{\frac{3}{2}}}\sum_{q,k}y_{tq}(X^*\mG_1^s\bv_i)_{k}\Big(\frac{\partial \mP}{\partial x_{qk}}\Big)^3\mP^{l-4},\nonumber\\
&\mathcal L_2:= \frac{1 }{2!N^{\frac{3}{2}}}\sum_{q,k}y_{tq}(X^*\mG_1^s\bv_i)_{k}\frac{\partial \mP}{\partial x_{qk}}\frac{\partial^2 \mP}{\partial x_{qk}^2}\mP^{l-3},\nonumber\\
&\mathcal L_3:=  \frac{1 }{3!N^{\frac{3}{2}}}\sum_{q,k}y_{tq}(X^*\mG_1^s\bv_i)_{k}\frac{\partial^3 \mP}{\partial x_{qk}^3}\mP^{l-2}.
\end{align*}
First, note that  $\mathcal L_1=\onh$, by using the facts $|\partial \mP/\partial x_{ik}|\prec 1$ (\cf (\ref{19072410})) and  $|(X^*\mG_1^s\bv_i)_{k}|\prec N^{-\frac12}$ (\cf (\ref{est_XG})), together with the Cauchy-Schwarz inequality. 

Second, as for $\mathcal L_2$, we use the formula of $\partial^2 \mP/\partial x_{qk}^2$ in \eqref{eq:Pd2}. Observe that, by the isotropic local laws  \eqref{est_DG} and \eqref{est_XG},  the terms in \eqref{eq:Pd2} can be  bounded  by either $O_\prec(N^{-1/2})$ or $O_\prec(\sn)$.
 More precisely, those $O_\prec(N^{-1/2})$ terms  possess one of the following forms 
\begin{align*}
\sn (X^*\mG_1^{b_1}\mb{\eta}_1)_k(X^*\mG_1^{b_2})_{kq}(\mG_1^{b_3}\mb{\eta}_2)_q, \qquad \sn (X^*\mG_1^{b_1}\mb{\eta}_1)_k(\mG_1^{b_2})_{qq}(X^*\mG_1^{b_3}\mb{\eta}_2)_k,
\end{align*} 
and  those $O_\prec(\sn)$ terms possess the common form $\sn (\mG_1^{b_1}\mb{\eta}_1)_q(X^*\mG_1^{b_2}X)_{kk}^a(\mG_1^{b_3}\mb{\eta}_2)_q$. Here, $\mb{\eta}_1,\mb{\eta}_2=\mb{y}_1,\mb{y}_2,$ or $\bv_i$,  $a=0,1$ and $b_1,b_2,b_3=1$ or $2$.
 The contribution from the $O_\prec(N^{-1/2})$ terms can be discussed similarly to $\mathcal{L}_1$. We thus omit the details. For the contribution from the  those  $O_\prec(\sn)$ terms,   we notice that it suffices to consider the bound of the following forms 
\begin{align}
\frac{1}{\sn}\Big(\sum_{q}y_{tq}(\mG_1^{b_1}\mb{\eta}_1)_q(\mG_1^{b_2}\mb{y}_b)_q(\mG_1^{b_3}\bv_i)_q\Big)\Big(\sum_k(X^*\mG_1^s\bv_i)_k(X^*\mG_1^{b_4}\bv_i)_k\big((X^*\mG_1^{b_5}X)_{kk}\big)^a\Big) \mP^{l-3},\quad a=0,1
\end{align}
which is $\onh$ by \rf{Importestsum1} and the isotropic local laws  \eqref{est_DG} and \eqref{est_XG}. Hence, $\mathcal L_2$ is also bounded by $\onh$.

Finally, for $\mathcal L_3$, we use the bound in (\ref{eq:bdPd3}). 
Then following the same argument as for $\mathcal L_1$, one can prove that $\mathcal L_3= O_\prec(N^{-\frac12})$. 
 To conclude, we have $h_{t,s}(0,3)=\onh$.

Hence, we conclude the proof of (2) in Lemma \ref{rankrlemh_ts}.

Before we proceed to the proof of the remainder term $\mathcal R_{t,s}$, let us comment that, using the same reasoning as we  previously did for $h_{t,s}(\alpha_1,\alpha_2)$ with $\alpha_1+\alpha_2\le 3$, one can also get
\begin{align}\label{eq:bdh4'}
N^{-2} \sum_{q,k} \Big|y_{tq} \frac{\partial^4 \big((X^*\mG_1^s\bv_i)_{k}\mP^{l-1}\big)}{\partial x_{qk}^4}\Big| = O_\prec(N^{-\frac12}),
\end{align}
which further implies that $h_{t,s}(\alpha_1,\alpha_2) = O_\prec(N^{-\frac12})$ for $\alpha_1+\alpha_2 =4$. The main tools are still Lemma \ref{Importantests} and the isotropic local laws \eqref{est_DG}  and \eqref{est_XG}. We omit the details. The necessary formulas for the fourth derivatives of $(X^*\mG_1^s)_{kj}$ and $\mP$ are recorded in the Appendix for the readers' convenience.

In the end, we prove  (3) of Lemma \ref{rankrlemh_ts}, i.e.,  we show the estimate of $\mathcal{R}_{t,s}$. By Lemma \ref{cumulantexpansion}, we can bound $\mathcal{R}_{t,s}$ by
\begin{align}
|\mathcal{R}_{t,s}|\leq \sn \sum_{q,k}\mathbb{E}\Big( & N^{-\frac{5}{2}}\sup_{|x_{qk}|\leq N^{-\frac{1}{2}+\epsilon}}\Big|  y_{tq}\frac{\partial^4 \big((X^*\mG_1^s\bv_i)_{k}\mP^{l-1}\big)}{\partial x_{qk}^4}\Big| \nonumber\\
&+N^{-\wt K}\sup_{x_{qk}\in\mathbb{R}}\Big| y_{tq}\frac{\partial^4 \big((X^*\mG_1^s\bv_i)_{k}\mP^{l-1}\big)}{\partial x_{qk}^4}\Big|\Big), \label{boundR_t}
\end{align}
for any sufficiently large constant $\wt K$. We evaluate the RHS of \rf{boundR_t} term by term. 

First, we claim that similarly to  \eqref{eq:bdh4'} we have 
\begin{align}
\sn\mathbb{E}\Big(N^{-\frac{5}{2}}\sup_{|x_{qk}|\leq N^{-\frac{1}{2}+\epsilon}}\sum_{q,k}\Big| y_{tq}\frac{\partial^4 \big((X^*\mG_1^s\bv_i)_{k}\mP^{l-1}\big)}{\partial x_{qk}^4}\Big|\Big)=O_\prec(N^{-\frac12}). \label{19072201}
\end{align}
Differently from  \eqref{eq:bdh4'}, in (\ref{19072201}), we actually consider a random matrix $\wt{X}$ with the $(q,k)$-th entry deterministic while all the other entries random.  Using a regular perturbation argument through the resolvent expansion,   one can show that replacing one random entry $x_{qk}$ in $X$ by any deterministic  number bounded by $N^{-1/2+\epsilon}$ while keeping all the other $X$ entries random will not change the isotropic local law. Then the isotropic local law together with the trivial deterministic bounds in \rf{trivialbd_G1}-\rf{trivialbd_XG1X} leads to (\ref{19072201}).

For the second term of \rf{boundR_t}, we simply use the crude deterministic bounds in  \rf{trivialbd_G1}-\rf{trivialbd_XG1X}. By choosing $\wt{K}$ sufficiently large, we can conclude that the second term in   \rf{boundR_t} is negligible. 
  
 This completes the proof of  Lemma \ref{rankrlemh_ts}.

\end{proof}

\appendix
\section{ Collection of derivatives} \label{s.derivative of G}
In this section, we summarize some derivatives that appear in the previous sections. And all these derivatives can be obtained by repeatedly applying the second identity in \rf{derivative} and chain rule. 
For convenience, we set 
\begin{align}
\mathcal{O}_l:=\{(o_1,\cdots,o_l):\exists i\in \{1,\cdots,l\}, o_i=0, \text{and }o_{j}=1,2  \text{ for all $j\neq i$}\} \label{def of O}
\end{align}

Below,  we first collect  the derivatives of $(X^*\mG_1^s\bv)_{k}$  for some deterministic unit vector $\bv$, 
which can be derived by using (\ref{derivative}) and the product rule.
The first derivative of $(X^*\mG_1^s\bv)_{k}$
\begin{align}
\label{deriXG^l}
 &\frac{\partial (X^*\mG_1^{s}\bv)_{k}}{\partial x_{qk}}=(\mG_1^s\bv)_{q}-\sum_{a=1}^2\sum_{\begin{subarray}{c}s_1,s_2\geq 1;\\ s_1+s_2=s+1 \end{subarray} }(X^*\mG_1^{s_1}\mathscr{P}_a^{qk}\mG_1^{s_2}\bv)_{k}.
\end{align}

The second derivative of  $(X^*\mG_1^s\bv)_{k}$
\begin{align} \label{eq:2ndXG}
\frac{\partial^2 (X^*\mG_1^s\bv)_{k}}{\partial x_{qk}^2}=&2\bigg(-\sum_{a=1}^2\sum_{\begin{subarray}{c}s_1,s_2\geq 1;\\ s_1+s_2=s+1 \end{subarray}}\Big(\mG_1^{s_1}\mathscr{P}_a^{qk}\mG_1^{s_2}\bv\Big)_{q}+\sum_{a_1,a_2=1}^2\sum_{\begin{subarray}{c}s_1, s_2, s_3\geq 1;\\ \sum_{i=1}^3s_i=s+2 \end{subarray}}\Big(X^*\mG_1^{s_1}\mathscr{P}_{a_1}^{qk}\mG_1^{s_2}\mathscr{P}_{a_2}^{qk}\mG_1^{s_3}\bv\Big)_{k}\nonumber\\
&
-\sum_{\begin{subarray}{c}s_1,s_2\geq 1;\\ s_1+s_2=s+1 \end{subarray}}\Big(X^*\mG_1^{s_1}\mathscr{P}_{0}^{qk}\mG_1^{s_2}\bv\Big)_{k}\bigg).
\end{align}

The third derivative of  $(X^*\mG_1^s\bv)_{k}$

\begin{align} 
\label{rankrorder3derivative}
\frac{\partial^3 (X^*\mG_1^s\bv)_{k}}{\partial x_{qk}^3}=& 6\bigg(\sum_{a_1,a_2=1}^2\sum_{\begin{subarray}{c}s_1,s_2,s_3\geq 1;\\ \sum_{i=1}^3s_i=s+2 \end{subarray}}\Big(\mG_1^{s_1}\mathscr{P}_{a_1}^{qk}\mG_1^{s_2}\mathscr{P}_{a_2}^{qk}\mG_1^{s_3}\bv\Big)_{q}-\sum_{\begin{subarray}{c}s_1,s_2\geq 1;\\ s_1+s_2=s+1 \end{subarray}}\Big(\mG_1^{s_1}\mathscr{P}_{0}^{qk}\mG_1^{s_2}\bv\Big)_{q} \nonumber\\
&-\sum_{a_1,a_2,a_3=1}^2\sum_{\begin{subarray}{c}s_1,\cdots,s_4\geq 1;\\ \sum_{i=1}^4s_i=s+3 \end{subarray}}\Big(X^*\Big(\prod_{i=1}^3(\mG_1^{s_i}\mathscr{P}_{a_i}^{qk})\Big)\mG_1^{s_4}\bv\Big)_{k}\nonumber\\
&+\sum_{(a_1,a_2)\in \mathcal{O}_2}\sum_{\begin{subarray}{c}s_1,\cdots,s_3\geq 1;\\ \sum_{i=1}^3s_i=s+2 \end{subarray}}\Big(X^*\Big(\prod_{i=1}^2(\mG_1^{s_i}\mathscr{P}_{a_i}^{qk})\Big)\mG_1^{s_3}\bv\Big)_{k}\bigg).
\end{align}

The fourth derivative of  $(X^*\mG_1^s\bv)_{k}$
\begin{align} 
\frac{\partial^4 (X^*\mG_1^s\bv)_{k}}{\partial x_{ik}^4}=&4!\bigg(-\sum_{a_1,a_2,a_3=1}^2\sum_{\begin{subarray}{c}s_1,\cdots,s_4\geq 1;\\ \sum_{i=1}^4s_i=s+3 \end{subarray}}\Big(\Big(\prod_{i=1}^3(\mG_1^{s_i}\mathscr{P}_{a_i}^{qk})\Big)\mG_1^{s_4}\bv\Big)_{q} \nonumber\\
&+\sum_{(a_1,a_2)\in \mathcal{O}_2}\sum_{\begin{subarray}{c}s_1,s_2,s_3\geq 1;\\ \sum_{i=1}^3s_i=s+2 \end{subarray}}\Big(\Big(\prod_{i=1}^2(\mG_1^{s_i}\mathscr{P}_{a_i}^{qk})\Big)\mG_1^{s_3}\bv\Big)_{q}\nonumber\\
&+\sum_{a_1,\cdots,a_4=1}^2\sum_{\begin{subarray}{c}s_1,\cdots,s_5\geq 1;\\ \sum_{i=1}^5s_i=s+4\end{subarray}}\Big(X^*
\Big(\prod_{i=1}^4(\mG_1^{s_i}\mathscr{P}_{a_i}^{qk})\Big)\mG_1^{s_5}\bv\Big)_{k} \nonumber\\
&-\sum_{(a_1,a_2,a_3)\in \mathcal{O}_3}\sum_{\begin{subarray}{c}s_1,\cdots,s_4\geq 1;\\ \sum_{i=1}^4s_i=s+3 \end{subarray}}\Big(X^*\Big(\prod_{i=1}^3(\mG_1^{s_i}\mathscr{P}_{a_i}^{qk})\Big)\mG_1^{s_4}\bv\Big)_{k}\nonumber\\
&+\sum_{\begin{subarray}{c}s_1,s_2,s_3\geq 1;\\ \sum_{i=1}^3s_i=s+2 \end{subarray}}\Big(X^*\Big(\prod_{i=1}^2(\mG_1^{s_i}\mathscr{P}_{0}^{qk})\Big)\mG_1^{s_3}\bv\Big)_{k}\bigg).
\end{align}

Next, for $\mathcal P$ defined in \eqref{rankrrepresentmP}, we collect its derivatives  $$ \frac{\partial^s \mP}{\partial x_{qk}^s}=\sn \sum_{b=1}^2\mb{y}_b^*\frac{\partial^s \mG_1^b}{\partial x_{qk}^s}\bv_i$$ for $1\le s \le 4$. For the first derivative, we have
\begin{align}\label{eq:Pd1}
\frac{\partial \mP}{\partial x_{qk}}
&=-\sn\sum_{b=1}^2\sum_{\begin{subarray}{c}b_1,b_2\geq 1;\\ b_1+b_2=b+1 \end{subarray}}\Big((\mb{y}_b^*\mG_1^{b_1})_q(X^*\mG_1^{b_2}\bv_i)_k+(X^*\mG_1^{b_1}\mb{y}_b)_k(\mG_1^{b_2}\bv_i)_q\Big).
\end{align}
By \rf{est_XG} and Remark \ref{boundrmk}, it is easy to see that 

\begin{align}
\Big|\frac{\partial \mP}{\partial x_{qk}}\Big|\prec 1. \label{19072410}
\end{align}

The second derivative of $\mP$ is
\begin{align}\label{eq:Pd2}
 \frac{\partial^2 \mP}{\partial x_{qk}^2}&=2\sn\sum_{b=1}^2\Big(\sum_{a_1,a_2=1}^2\sum_{\begin{subarray}{c}b_1,b_2,b_3\geq 1;\\ \sum_{i=1}^3b_i=b+2 \end{subarray}}\mb{y}_b^*\Big(\prod_{j=1}^2(\mG_1^{b_j}\mathscr{P}_{a_j}^{qk})\Big)\mG_1^{b_3}\bv_i-\sum_{\begin{subarray}{c}b_1,b_2\geq 1;\\ b_1+b_2=b+1 \end{subarray}}\mb{y}_b^*\mG_1^{b_1}\mathscr{P}_0^{qk}\mG_1^{b_2}\bv_i\Big).
\end{align}
Recalling the definition of $\mathscr{P}^{qk}_i$ for $i=0,1,2$ in \rf{P012}, one observes that all terms in the parenthesis admit  one of the following forms 
\begin{align*}
(X^*\mG_1^{b_1}\mb{\eta}_1)_k(X^*\mG_1^{b_2})_{kq}(\mG_1^{b_3}\mb{\eta}_2)_q \qquad (X^*\mG_1^{b_1}\mb{\eta}_1)_k(\mG_1^{b_2})_{qq}(X^*\mG_1^{b_3}\mb{\eta}_2)_k \qquad 
(\mG_1^{b_1}\mb{\eta}_1)_q(X^*\mG_1^{b_2}X)_{kk}^a(\mG_1^{b_3}\mb{\eta}_2)_q
\end{align*}
for $\mb{\eta}_1,\mb{\eta}_2=\mb{y}_1,\mb{y}_2,\bv_i$, $a=0,1$ and $b_1,b_2,b_3=1,2$, which are bounded by $\onh$, $\onh$ and $O_\prec(1)$ respectively, in light of \rf{est_XG} and Remark \ref{boundrmk}. Therefore, combining with the prefactor $\sn$ in (\ref{eq:Pd2}), we get the bound 
\begin{align}
\Big| \frac{\partial^2 \mP}{\partial x_{qk}^2}\Big|\prec\sqrt{N}.  \label{19072411}
\end{align}

The third derivative of $\mP$ is
\begin{align}\label{eq:Pd3} 
\frac{\partial^3 \mP}{\partial x_{qk}^3}
&=-6\sn \sum_{b=1}^2\bigg(\sum_{a_1,a_2,a_3=1}^2\sum_{\begin{subarray}{c}b_1,\cdots,b_4\geq 1;\\\sum_{i=1}^4 b_i=b+3 \end{subarray}}\Big(\mb{y}_b^*\Big(\prod_{j=1}^3(\mG_1^{b_j}\mathscr{P}_{a_j}^{qk})\Big)\mG_1^{b_4}\bv_i\Big)\nonumber\\
&\quad-\sum_{(a_1,a_2)\in \mathcal{O}_2}\sum_{\begin{subarray}{c}b_1,b_2,b_3\geq 1;\\ \sum_{i=1}^3b_i=b+2 \end{subarray}}\Big(\mb{y}_b^*\Big(\prod_{j=1}^2(\mG_1^{b_j}\mathscr{P}_{a_j}^{qk})\Big)\mG_1^{b_3}\bv_i\Big)\bigg).
\end{align}
By plugging the definition of $\mathscr{P}^{qk}_a$ in (\ref{P012}), one can see that all summands  above contain at least one quadratic form of  $(X^*\mG_1^a)$ for some $a\geq 1$,  which  by \rf{est_XG} will contribute a $\onh$ factor. Using this fact, one can easily show the crude bound 

\begin{align}
\Big|\frac{\partial^3 \mP}{\partial x_{qk}^3}\Big|\prec 1.  \label{eq:bdPd3}
\end{align}
The fourth derivative of $\mP$ is
\begin{align}\label{eq:Pd4}
 \frac{\partial^4 \mP}{\partial x_{qk}^4}
&=4!\sn \sum_{b=1}^2\bigg(\sum_{a_1,\cdots,a_4=1}^2\sum_{\begin{subarray}{c}b_1,\cdots,b_5\geq 1;\\\sum_{i=1}^5 b_i=b+4 \end{subarray}}\Big(\mb{y}_b^*
\Big(\prod_{j=1}^4(\mG_1^{b_j}\mathscr{P}_{a_j}^{qk})\Big)\mG_1^{b_5}\bv_i\Big)\nonumber\\
&\quad-\sum_{(a_1,a_2,a_3)\in \mathcal{O}_3}\sum_{\begin{subarray}{c}b_1,\cdots,b_4\geq 1;\\\sum_{i=1}^4 b_i=b+3 \end{subarray}}\Big(\mb{y}_b^*\Big(\prod_{j=1}^3(\mG_1^{b_j}\mathscr{P}_{a_j}^{qk})\Big)\mG_1^{b_4}\bv_i\Big)\nonumber\\
&\quad+\sum_{\begin{subarray}{c}b_1,b_2,b_3\geq 1;\\ \sum_{i=1}^3b_i=b+2 \end{subarray}}\Big(\mb{y}_b^*\Big(\prod_{j=1}^2(\mG_1^{b_j}\mathscr{P}_{0}^{qk})\Big)\mG_1^{b_3}\bv_i\Big)\bigg). 
\end{align}

\bibliographystyle{unsrt}

\end{document}